\documentclass[11pt,a4paper,abstract=true,parskip=half-]{scrartcl}
\usepackage[english]{babel}
\usepackage[utf8]{inputenc}
\usepackage[T1]{fontenc}
\usepackage{lmodern}
\usepackage[tracking,kerning]{microtype}
\usepackage{mathtools}
\usepackage{amssymb}
\usepackage[shortlabels]{enumitem}
\usepackage[thmmarks,amsmath]{ntheorem}
\usepackage{dsfont}
\usepackage{hyperref}
\usepackage{doi}

\theoremseparator{.}
\newtheorem{theorem}{Theorem}[section]
\newtheorem{lemma}[theorem]{Lemma}
\newtheorem{proposition}[theorem]{Proposition}

\theorembodyfont{\upshape}
\newtheorem{definition}[theorem]{Definition}
\newtheorem{remark}[theorem]{Remark}
\newtheorem*{hypoc_data}{Data conditions}
\newtheorem*{hypoc_potential}{Hypocoercivity assumptions}
\newtheorem*{assumption}{Assumption}

\theoremstyle{nonumberbreak}
\theoremseparator{:}
\theoremsymbol{\ensuremath{\square}}
\newtheorem{proof}{Proof}

\newcommand{\nref}[2]{\hyperlink{#1}{\upshape\textbf{#2}}}

\newcommand*{\R}{\mathbb{R}}
\newcommand*{\N}{\mathbb{N}}
\newcommand*{\eps}{\varepsilon}
\newcommand*{\supp}{\operatorname{supp}}
\newcommand*{\tr}{\operatorname{tr}}
\newcommand*{\defeq}{\mathrel{\vcenter{\baselineskip0.5ex \lineskiplimit0pt
			\hbox{\scriptsize.}\hbox{\scriptsize.}}}%
	=}
\newcommand*{\eqdef}{=\mathrel{\vcenter{\baselineskip0.5ex \lineskiplimit0pt
			\hbox{\scriptsize.}\hbox{\scriptsize.}}}}

\numberwithin{equation}{section}

\title{Essential m-dissipativity and hypocoercivity of Langevin dynamics with multiplicative noise}
\author{Alexander Bertram%
	\thanks{Department of Mathematics, TU Kaiserslautern, PO box 3049, 67653 Kaiserslautern, Germany}
	\textsuperscript{,}%
	\thanks{krampe@mathematik.uni-kl.de (corresponding author)}
	\and
	Martin Grothaus%
	\footnotemark[1]
	\textsuperscript{,}%
	\thanks{grothaus@mathematik.uni-kl.de}
}

\begin{document}
\maketitle
\begin{abstract}
	We provide a complete elaboration of the $L^2$-Hilbert space hypocoercivity theorem for the degenerate Langevin dynamics with multiplicative noise, studying the longtime behavior of the strongly continuous contraction semigroup solving the abstract Cauchy problem for the associated backward Kolmogorov operator. Hypocoercivity for the Langevin dynamics with constant diffusion matrix was proven previously by Dolbeault, Mouhot and Schmeiser in the corresponding Fokker-Planck framework, and made rigorous in the Kolmogorov backwards setting by Grothaus and Stilgenbauer. We extend these results to weakly differentiable diffusion coefficient matrices, introducing multiplicative noise for the corresponding stochastic differential equation. The rate of convergence is explicitly computed depending on the choice of these coefficients and the potential giving the outer force. In order to obtain a solution to the abstract Cauchy problem, we first prove essential self-adjointness of non-degenerate elliptic Dirichlet operators on Hilbert spaces, using prior elliptic regularity results and techniques from Bogachev, Krylov and Röckner. We apply operator perturbation theory to obtain essential m-dissipativity of the Kolmogorov operator, extending the m-dissipativity results from Conrad and Grothaus. We emphasize that the chosen Kolmogorov approach is natural, as the theory of generalized Dirichlet forms implies a stochastic representation of the Langevin semigroup as the transition kernel of a diffusion process which provides a martingale solution to the Langevin equation with multiplicative noise. Moreover, we show that even a weak solution is obtained this way.
\end{abstract}


\subsection*{Acknowledgment}
This version of the article has been accepted for publication, after peer review
but is not the Version of Record and does not reflect post-acceptance improvements, or any
corrections.

The Version of Record is available online at: \url{https://doi.org/10.1007/s00028-022-00773-y}

\section{Introduction}

We study the exponential decay to equilibrium of Langevin dynamics with multiplicative noise.
The corresponding evolution equation is given by the following  stochastic differential equation
on $\R^{2d}$, $d\in\N$, as
\begin{equation}\label{eq:sde}
\begin{aligned}
	dX_t &= V_t\,\mathrm{d}t, \\
	dV_t &= b(V_t)\mathrm{d}t-\nabla\Phi(X_t)\,\mathrm{d}t + \sqrt{2}\sigma(V_t)\,\mathrm{d}B_t,
\end{aligned}
\end{equation}
where $\Phi:\R^d\to\R$ is a suitable potential whose properties are specified later,
$B=(B_t)_{t\geq 0}$ is a standard $d$-dimensional Brownian motion, $\sigma:\R^d\to\R^{d\times d}$
a variable diffusion matrix with at least weakly differentiable coefficients, 
and $b:\R^d\to\R^d$ given by
\[
	b_i(v)=\sum_{j=1}^d \partial_j a_{ij}(v) -a_{ij}(v)v_j,
\]
where $a_{ij}=\Sigma_{ij}$ with $\Sigma= \sigma\sigma^T$.

This equation describes the evolution of a particle described via its position $(X_t)_{t\geq 0}$
and velocity $(V_t)_{t\geq 0}$ coordinates, which is subject to friction, stochastic perturbation depending on its velocity, and some outer force $\nabla\Phi$. To simplify notation, we split $\R^{2d}$ into the two components $x,v\in\R^d$
corresponding to position and velocity respectively. This extends to differential operators $\nabla_x,\nabla_v$,
and the Hessian matrix $H_v$.

Using Itô's formula, we obtain the associated Kolmogorov operator $L$ as
\begin{equation}\label{eq:langevin-op}
	L=\tr\left(\Sigma H_v \right)+ b(v)\cdot\nabla_v +v\cdot\nabla_x -\nabla\Phi(x)\cdot\nabla_v.
\end{equation}
Here $a\cdot b$ or alternatively $(a,b)_{\mathrm{euc}}$ denotes the standard inner product of $a,b\in\R^d$.
We introduce the measure $\mu=\mu_{\Sigma,\Phi}$ on $(\R^{2d},\mathcal{B}(\R^{2d}))$ as
\[
	\mu_{\Sigma,\Phi} = (2\pi)^{-\frac{d}2} \mathrm{e}^{-\Phi(x)-\frac{v^2}2}\,\mathrm{d}x\otimes\mathrm{d}v
	\eqdef \mathrm{e}^{-\Phi(x)}\otimes\nu,
\]
i.e. $\nu$ is the normalized standard Gaussian measure on $\R^d$. We consider the operator $L$ on the Hilbert space $H\defeq L^2(\R^{2d},\mu)$.

We note that the results below on exponential convergence to equilibrium can also be translated to a corresponding Fokker-Planck setting, with the differential operator $L^\mathrm{FP}$ given as the adjoint, restricted to sufficiently smooth functions, of $L$ in $L^2(\R^{2d},\mathrm{d}(x,v))$.
The considered Hilbert space there is $\tilde{H}\defeq L^2(\R^{2d},\tilde{\mu})$, where
\[
	\tilde{\mu}\defeq  (2\pi)^{-\frac{d}2} \mathrm{e}^{\Phi(x)+\frac{v^2}2}\,\mathrm{d}x\otimes\mathrm{d}v.
\]
Indeed, this is the space in which hypocoercivity of the kinetic Fokker-Planck equation associated with the classical Langevin dynamics was proven in \cite{DMS15}. The rigorous  connection to the Kolmogorov backwards setting considered throughout this paper and convergence behaviour of solutions to the abstract Cauchy problem $\partial_t f(t)=L^\mathrm{FP}f(t)$ are discussed in Section~\ref{subsec:fokker-planck}.

The concept of hypocoercivity was first introduced in the memoirs of Cédric Villani (\cite{Villani}), which is recommended as further literature to the interested reader. The approach we use here was introduced algebraically by Dolbeault, Mouhot and Schmeiser (see \cite{DMS09} and \cite{DMS15}), and then made rigorous including domain issues in \cite{GS14} by Grothaus and Stilgenbauer, where it was applied to show exponential convergence to equilibrium of a Fiber laydown process on the unit sphere. This setting was further generalized by Wang and Grothaus in \cite{GW}, where the coercivity assumptions involving in part the classical Poincaré inequality for Gaussian measures were replaced by weak Poincaré inequalities, allowing for more general measures for both the spatial and the velocity component. In this case, the authors still obtained explicit, but subexponential rates of convergence. On the other hand, the stronger notion of hypercontractivity was explored in \cite{W} on general separable Hilbert spaces without the necessity to explicitly state the invariant measure. The specific case of hypocoercivity for Langevin dynamics on the position space $\R^d$ has been further explored in \cite{GS16} and serves as the basis for our hypocoercivity result. However, all of these prior results assume the diffusion matrix to be constant, while we allow for velocity-dependent coefficients. 

In contrast to \cite{GS16}, we do not know if our operator $(L,C_c^\infty(\R^{2d}))$ is essentially m-dissipative, and are therefore left to prove that first. This property of the Langevin operator has been shown by Helffer and Nier in \cite{HN2005} for smooth potentials and generalized to locally Lipschitz-continuous potentials by Conrad and Grothaus in \cite[Corollary 2.3]{CG10}. However, a corresponding result for a non-constant second order coefficient matrix $\Sigma$ is not known to the authors.

Moreover, the symmetric part $S$ of our operator $L$ does not commute with the linear operator $B$ as in \cite{GS16},
hence the boundedness of the auxiliary operator $BS$ needs to be shown in a different way, which we do in Proposition~\ref{prop:aux-bound-sym}.

In Theorem~\ref{thm:ess-m-diss}, we show under fairly light assumptions on the coefficients and the potential that the operator $(L,C_c^\infty(\R^{2d}))$ is essentially m-dissipative and therefore generates a strongly continuous contraction semigroup on $H$. The proof is given in Section 4 and follows the main ideas as in the proof of \cite[Theorem 2.1]{CG10}, where a corresponding result for $\Sigma=I$ was obtained.

For that proof we rely on perturbation theory of m-dissipative operators, starting with essential m-dissipativity of the symmetric part of $L$. To that end, we state an essential self-adjointness result for a set of non-degenerate elliptic Dirichlet differential operators $(S,C_c^\infty(\R^d))$ on $L^2$-spaces
where the measure is absolutely continuous wrt.~the Lebesgue measure. This result is stated in Theorem~\ref{thm:ess-self-adjoint} and combines regularity results from \cite{BKR01} and \cite{BGS13} with the approach to show essential self-adjointness from \cite{BKR97}.

Finally, our main hypocoercivity result reads as follows:
\begin{theorem}\label{thm:main-result}
	Let $d\in\N$. Assume that $\Sigma:\R^d\to\R^{d\times d}$ is a symmetric matrix of coefficients
	$a_{ij}:\R^d\to\R$ which is uniformly strictly elliptic with ellipticity constant $c_\Sigma$.
	Moreover, let each $a_{ij}$ be bounded and locally Lipschitz-continuous,
	hence $a_{ij}\in H_{\mathrm{loc}}^{1,p}(\R^d,\nu)\cap L^\infty(\R^d)$ for each $p\geq 1$.
	Assume the growth behaviour of $\partial_k a_{ij}$ for all $1\leq k\leq d$ to be bounded either by
	\[
		|\partial_k a_{ij}(v)|\leq M(1+|v|)^\beta
	\]
	for $\nu$-almost all $v\in\R^d$ and some $M<\infty$, $\beta\in (-\infty,0]$ or by
	\[
		|\partial_k a_{ij}(v)|\leq M(\mathds{1}_{B_1(0)}(v)+|v|^\beta)
	\]
	for $\nu$-almost all $v\in\R^d$ and some $M<\infty$, $\beta\in (0,1)$.
	Define $N_\Sigma$ in the first case as $N_\Sigma\defeq \sqrt{M_\Sigma^2+(B_\Sigma\vee M)^2}$
		and in the second case as $N_\Sigma\defeq \sqrt{M_\Sigma^2+B_\Sigma^2+dM^2}$, where
	\begin{align*}
		M_\Sigma &\defeq \max\{ \|a_{ij}\|_\infty\mid 1\leq i,j \leq d \}\quad\text{ and }\\
		B_\Sigma &\defeq \max\left\{|\partial_j a_{ij}(v)| : v\in\overline{B_1(0)},\ 1\leq i,j \leq d \right\}.
	\end{align*}
	Let further $\Phi:\R^d\to\R$ be bounded from below, satisfy $\Phi\in C^2(\R^d)$ and that $\mathrm{e}^{-\Phi(x)}\,\mathrm{d}x$
	is a probability measure on $(\R^d,\mathcal{B}(\R^d))$ which satisfies a Poincaré inequality of the form
	\[
		\|\nabla f\|_{L^2(\mathrm{e}^{-\Phi(x)}\,\mathrm{d}x)}^2
		\geq \Lambda\left\| f-\int_{R^d}f\mathrm{e}^{-\Phi(x)}\,\mathrm{d}x \right\|_{L^2(\mathrm{e}^{-\Phi(x)}\,\mathrm{d}x)}^2
	\]
	for some $\Lambda\in (0,\infty)$ and all $f\in C_c(\R^d)$. Furthermore assume the existence of a constant $c<\infty$ such that 
	\[
		|H\Phi(x)|\leq c(1+|\nabla\Phi(x)|)\quad\text{ for all }x\in\R^d,
	\]
	where $H$ denotes the Hessian matrix and $|H\Phi|$ the Euclidian matrix norm.
	If $\beta >-1$, then also assume that there are constants $N<\infty$, $\gamma < \frac{2}{1+\beta}$ such that
	\[
		|\nabla\Phi(x)|\leq N(1+|x|^\gamma)\qquad \text{ for all } x\in\R^d.
	\]
	
	Then the Langevin operator $(L,C_c^\infty(\R^{2d}))$ as defined in \eqref{eq:langevin-op} is closable on $H$ and its closure $(L,D(L))$ generates a strongly continuous contraction semigroup $(T_t)_{t\geq 0}$ on $H$. Further, it holds that for each $\theta_1\in (1,\infty)$,
	there is some $\theta_2\in (0,\infty)$ such that
	\[
		\left\| T_tg- (g,1)_H \right\|_H \leq \theta_1 \mathrm{e}^{-\theta_2 t}\left\| g-(g,1)_H  \right\|_H
	\]
	for all $g\in H$ and all $t\geq 0$. In particular, $\theta_2$ can be specified as
	\[
		\theta_2=\frac{\theta_1-1}{\theta_1}\frac{c_\Sigma}{n_1+n_2N_\Sigma+n_3N_\Sigma^2},
	\]
	and the coefficients $n_i\in(0,\infty)$ only depend on the choice of $\Phi$.
\end{theorem}

Finally, our main results may be summarized by the following list:
\begin{itemize}[$\bullet$]
	\item Essential m-dissipativity (equivalently essential self-adjointness) of non-degenerate elliptic Dirichlet differential operators with domain $C_c^\infty(\R^d)$
		on Hilbert spaces with measure absolutely continuous wrt.~the $d$-dimensional Lebesgue measure is proved, see Theorem~\ref{thm:ess-self-adjoint}.
	\item Essential m-dissipativity of the backwards Kolmogorov operator $(L,C_c^\infty(\R^d))$ associated with the Langevin equation
		with multiplicative noise \eqref{eq:sde} on the Hilbert space $H$ under weak assumptions on the coefficient matrix $\Sigma$ and the potential $\Phi$, in particular not requiring smoothness, is shown, see Theorem~\ref{thm:ess-m-diss}.
	\item Exponential convergence to a stationary state of the corresponding solutions to the abstract Cauchy problem
		$\partial_t u(t) = Lu(t)$, see \eqref{eq:cauchy-kol} on the Hilbert space $H$ with explicitly computable rate of convergence, as stated in Theorem~\ref{thm:main-result}, is proved.
	\item Adaptation of this convergence result to the equivalent formulation as a Fokker-Planck PDE on the appropriate Hilbert space $\tilde{H}\defeq L^2(\R^{2d},\tilde{\mu})$ is provided. In particular, this yields exponential convergence of the solutions to the abstract Fokker-Planck Cauchy problem $\partial_t u(t)=L^\mathrm{FP}u(t)$, with $L^\mathrm{FP}$ given by \eqref{eq:fokker-planck-op}, to a stationary state, see Section~\ref{subsec:fokker-planck}.
	\item A stochastic interpretation of the semigroup as a transition kernel for a diffusion process is worked out.
		Moreover, we prove this diffusion process to be a weak solution to the Langevin SDE \eqref{eq:sde} and derive for it strong mixing properties with explicit rates of convergence, see Section~\ref{subsec:stochastics}.
\end{itemize}

\section{The abstract hypocoercivity setting}

We start by recalling some basic facts about closed unbounded operators on Hilbert spaces:
\begin{lemma}\label{lem:adjoint_operator}
	Let $(T,D(T))$ be a densely defined linear operator on $H$ and
	let $L$ be a bounded linear operator with domain $H$.
	\begin{enumerate}[(i)]
		\item The adjoint operator $(T^*,D(T^*))$ exists and is closed.
			If $D(T^*)$ is dense in $H$, then $(T,D(T))$ is closable
			and for the closure $(\overline{T},D(\overline{T}))$ it holds
			$\overline{T}=T^{**}$.
		\item $L^*$ is bounded and $\|L^*\|=\|L\|$.
		\item If $(T,D(T))$ is closed, then $D(T^*)$ is automatically dense in $H$.
			Consequently by (i), $T=T^{**}$.
		\item Let $(T,D(T))$ be closed. Then the operator $TL$ with domain
			\[
				D(TL)=\{ f\in H\mid Lf\in D(T) \}
			\]
			is also closed.
		\item $LT$ with domain $D(T)$ need not be closed, however
			\[
				(LT)^*=T^*L^*.
			\]
	\end{enumerate}
\end{lemma}
Let us now briefly state the abstract setting for the hypocoercivity method as in \cite{GS14}.
\begin{hypoc_data}[\hypertarget{ass:data}{D}]
	We require the following conditions which are henceforth assumed without further mention.
	\begin{enumerate}[(D1)]
		\item \emph{The Hilbert space:} Let $(E,\mathcal{F},\mu)$ be some probability space
			and define $H$ to be $H=L^2(E,\mu)$ equipped with the standard inner product $(\cdot,\cdot)_H$.
		\item \emph{The $C_0$-semigroup and its generator:} $(L,D(L))$ is some linear operator on $H$ generating
			a strongly continuous contraction semigroup $(T_t)_{t\geq0}$.
		\item \emph{Core property of $L$:} Let $D\subset D(L)$ be a dense subspace of $H$
			which is a core for $(L,D(L))$.
		\item \emph{Decomposition of $L$:} Let $(S,D(S)))$ be symmetric,
			$(A,D(A))$ be closed and antisymmetric on $H$ such that
			$D\subset D(S)\cap D(A)$ as well as $L|_D=S-A$.
		\item \emph{Orthogonal projections:} Let $P:H\to H$ be an orthogonal projection satisfying
			$P(H)\subset D(S),\, SP=0$ as well as $P(D)\subset D(A),\, AP(D)\subset D(A)$.
			Moreover, let $P_S:H\to H$ be defined as
			\[
				P_Sf\defeq Pf+(f,1)_H,\qquad f\in H.
			\] 
		\item \emph{Invariant measure:} Let $\mu$ be invariant for $(L,D)$ in the sense that
			\[
				(Lf,1)_H = \int_E Lf\,\mathrm{d}\mu=0\qquad \text{ for all }f\in D.
			\]
		\item \emph{Conservativity:} It holds that $1\in D(L)$ and $L1=0$.
	\end{enumerate}
\end{hypoc_data}
Since $(A,D(A))$ is closed, $(AP, D(AP))$ is also closed and densely defined.
Hence by von Neumann's theorem, the operator
\[
I+(AP)^*(AP): D((AP)^*AP)\to H,
\]
where $D((AP)^*AP)=\{f\in D(AP)\mid APf\in D((AP)^*) \}$,
is bijective and admits a bounded inverse.
We therefore define the operator $(B,D((AP)^*))$ via
\[
B\defeq (I+(AP)^*AP)^{-1}(AP)*
\]
Then $B$ extends to a bounded operator on $H$.

As in the given source, we also require the following assumptions:
\begin{assumption}[\hypertarget{ass:algebraic}{H1}]
	\emph{Algebraic relation:} It holds that $PAP|_D=0$.
\end{assumption}
\begin{assumption}[\hypertarget{ass:micro-coerc}{H2}]
	\emph{Microscopic coercivity:} There exists some $\Lambda_m>0$ such that
	\[
		-(Sf,f)_H \geq \Lambda_m\|(I-P_S)f \|^2\qquad \text{ for all }f\in D.
	\]
\end{assumption}
\begin{assumption}[\hypertarget{ass:macro-coerc}{H3}]
	\emph{Macroscopic coercivity:} Define $(G,D)$ via $G=PA^2P$ on $D$.
	Assume that $(G,D)$ is essentially self-adjoint on $H$.
	Moreover, assume that there is some $\Lambda_M>0$ such that
	\[
		\|APf\|^2 \geq \Lambda_M \|Pf\|^2 \qquad\text{ for all }f\in D.
	\]
\end{assumption}
\begin{assumption}[\hypertarget{ass:aux-bound}{H4}]
	\emph{Boundedness of auxiliary operators:}
	The operators $(BS,D)$ and $(BA(I-P),D)$ are bounded
	and there exist constants $c_1,c_2<\infty$ such that
	\[
		\|BSf\|\leq c_1\|(I-P)f\| \quad\text{ and }\quad
		\|BA(I-P)f\| \leq c_2 \|(I-P)f\|
	\]
	hold for all $f\in D$. 
\end{assumption}

We now state the central abstract hypocoercivity theorem as in \cite{GS14}:
\begin{theorem}\label{thm:hypoc}
	Assume that (D) and (H1)-(H4) hold. Then there exist strictly positive
	constants $\kappa_1,\kappa_2<\infty$ which are explicitly computable
	in terms of $\Lambda_m,\Lambda_M, c_1$ and $c_2$ such that
	for all $g\in H$ we have
	\[
		\|T_tg-(g,1)_H\|\leq \kappa_1\mathrm{e}^{-\kappa_2t} \|g-(g,1)_H\|
		\quad\text{ for all }t\geq0.
	\]
	More specifically, if there exist $\delta>0$, $\eps\in(0,1)$ and $0<\kappa<\infty$ such that
	for all $g\in D(L)$, $t\geq 0$, it holds
	\begin{equation}\label{eq:conv-rate}
	\begin{aligned}
		\kappa\|f_t\|^2\leq
		\left(\Lambda_m-\eps(1+c_1+c_2)\left(1+\frac1{2\delta}\right) \right)&\|(I-P)f_t\|^2\\
		+ \eps\left(\frac{\Lambda_M}{1+\Lambda_M}-(1+c_1+c_2)\frac\delta2 \right)&\|Pf_t\|^2,
	\end{aligned}
	\end{equation}
	where $f_t\defeq T_tg-(g,1)_H$, then the constants $\kappa_1$ and $\kappa_2$ are given by
	\[
		\kappa_1=\sqrt{\frac{1+\eps}{1-\eps}},\qquad \kappa_2=\frac\kappa{1+\eps}.
	\]
\end{theorem}

In order to prove (H4), we will make use of the following result:
\begin{lemma}\label{lem:auxiliary-bound}
	Assume (H3). Let $(T,D(T))$ be a linear operator with $D\subset D(T)$ and assume $AP(D)\subset D(T^*)$.
	Then
	\[
		(I-G)(D)\subset D((BT)^*) \quad\text{ with }\quad (BT)^*(I-G)f = T^*APf,\quad f\in D.
	\]
	If there exists some $C<\infty$ such that
	\begin{equation}\label{eq:adjoint_bound}
		\|(BT)^* g\| \leq C\|g\|\qquad\text{ for all }g=(I-G)f,\quad f\in D,
	\end{equation}
	then $(BT,D(T))$ is bounded and its closure $(\overline{BT})$ is a bounded operator on $H$ with
	$\| \overline{BT}\|=\| (BT)^*\|$.
	
	In particular, if $(S,D(S))$ and $(A,D(A))$ satisfy these assumptions,
	the corresponding inequalities in (H4)
	are satisfied with $c_1=\| (BS)^*\|$ and $c_2=\| (BA)^*\|$.
\end{lemma}
\begin{proof}
	Let $h\in D((AP)^*)$ and $f\in D$. Set $g=(I-G)f$.
	By the representation of $B$ on $D((AP)^*)$ together with self-adjointness of $(I+(AP)^*AP)^{-1}$
	and $D\subset D(AP)$, we get	
	\[
		(h,B^*g)_H = (Bh, (I-G)f)_H = ((AP)^*h, f)_H = (h, APf)_H.
	\]
	So $B^*g = APf\in D(T^*)$. By Lemma~\ref{lem:adjoint_operator}~(v), $((BT)^*,D((BT)^*))=(T^*B^*, D(T^*B^*))$, which implies $(BT)^*g=T^*B^*g=T^*APf$.
	
	By essential self-adjointness and hence essential m-dissipativity of $G$, $(I-G)(D)$ is dense in $H$.
	Therefore by \eqref{eq:adjoint_bound}, the closed operator $((BT)^*,D((BT)^*))$ is a bounded operator on $H$.
	Since $(BT,D(T))$ is densely defined, by Lemma~\ref{lem:adjoint_operator}~(i) and (ii), it is closable
	with $\overline{BT}=(BT)^{**}$, which is a bounded operator on $H$ with the stated norm.
	
	The last part follows directly by $Sf=S(I-P)f$ for $f\in D$.
\end{proof}

\section{Hypocoercivity for Langevin dynamics with multiplicative noise}
As stated in the introduction, the aim of this section is to prove exponential convergence to equilibrium
of the semigroup solving the abstract Kolmogorov equation corresponding to the Langevin equation with multiplicative noise \eqref{eq:sde}.

We remark that most of the conditions are verified analogously to \cite{GS16}, the main difference being the proof of essential m-dissipativity for the operator $(L,C_c^\infty(\R^{2d}))$ as well as the first inequality in (H4).
Nevertheless, some care has to be taken whenever $S$ is involved, as it doesn't preserve regularity to the same extent as in the given reference.

\subsection{The data conditions}\label{subsec:data}
We start by introducing the setting and verifying the data conditions \nref{ass:data}{(D)}.
The notations introduced in this part will be used for the remainder of the section without further mention.

Let $d\in\N$ and set the state space as $E=\R^{2d}$, $\mathcal{F}=\mathcal{B}(\R^{2d})$.
In the following, the first $d$ components of $E$ will be written as $x$, the latter $d$ components as $v$.
Let $\nu$ be the normalised Gaussian measure on $\R^d$ with mean zero and covariance matrix $I$, i.e.
\[
\nu(A) = \int_{A} (2\pi)^{-\frac{d}2}\ \mathrm{e}^{-\frac{x^2}2}\,\mathrm{d}x. 
\]
\begin{assumption}[\hypertarget{ass:potential}{P}]
	The potential $\Phi:\R^d\to\R$ is assumed to depend only on the position variable $x$ and to be locally Lipschitz-continuous.
	We further assume $\mathrm{e}^{-\Phi(x)}\,\mathrm{d}x$ to be a probability measure on $(\R^d, \mathcal{B}(\R^d))$.
\end{assumption}
Note that the first part implies $\Phi\in H_\text{loc}^{1,\infty}(\R^d)$. Moreover, $\Phi$ is differentiable $\mathrm{d}x$-a.e. on $\R^d$,
such that the weak gradient and the derivative of $\Phi$ coincide $\mathrm{d}x$-a.e. on $\R^d$. In the following, we fix a version
of $\nabla \Phi$.

The probability measure $\mu$ on $(E,\mathcal{F})$ is then given by $\mu = \mathrm{e}^{-\Phi(x)}\,\mathrm{d}x\otimes\nu$, and we set
$H\defeq L^2(E,\mu)$, which satisfies condition (D1). Next we assume the following about $\Sigma=(a_{ij})_{1\leq i,j\leq d}$ with $a_{ij}:\R^d\to\R$:
\begin{assumption}[\hypertarget{ass:ellipticity}{$\Sigma$1}]
	$\Sigma$ is symmetric and uniformly strictly elliptic, i.e.~there is some $c_\Sigma>0$ such that
		\[
			(y,\Sigma(v) y) \geq c_\Sigma\cdot |y|^2\quad\text{ for all }y,v\in\R^d.
		\]
\end{assumption}
\begin{assumption}[\hypertarget{ass:coeff-derivatives}{$\Sigma$2}]
	There is some $p>d$ such that for all $1\leq i,j\leq d$, it holds that $a_{ij}\in H_\text{loc}^{1,p}(\R^d,\nu)\cap L^\infty(\R^d)$.
	Additionally, $a_{ij}$ is locally Lipschitz-continuous for all $1\leq i,j\leq d$.
\end{assumption}
Additionally, we will consider one of the following conditions on the growth of the partial derivatives:
\begin{assumption}[\hypertarget{ass:coeff-growth-one}{$\Sigma$3}]
	There are constants $0\leq M<\infty$, $-\infty< \beta \leq 0$ such that for all $1\leq i,j,k\leq d$
		\[
			|\partial_k a_{ij}(v)|\leq M(1+|v|)^\beta\quad\text{ for $\nu$-almost all }v\in\R^d.
		\]
\end{assumption}
\begin{assumption}[\hypertarget{ass:coeff-growth-two}{$\Sigma$3$'$}]
	There are constants $0\leq M<\infty$, $0< \beta <1$ such that for all $1\leq i,j,k\leq d$
	\[
	|\partial_k a_{ij}(v)|\leq M(\mathds{1}_{B_1(0)}(v)+|v|^\beta)\quad\text{ for $\nu$-almost all }v\in\R^d.
	\]
\end{assumption}
	We note that any of these imply $\partial_j a_{ij}\in L^2(\R^d,\nu)$ for all $1\leq i,j\leq d$.
\begin{definition}\label{def:coeff-const}
	Let $\Sigma$ satisfy \hyperlink{ass:coeff-derivatives}{($\Sigma$2)}. Then we set
	\begin{align*}
		M_\Sigma &\defeq \max\{ \|a_{ij}\|_\infty: 1\leq i,j \leq d \} \qquad\text{ and }\\
		B_\Sigma &\defeq \max\{ |\partial_j a_{ij}(v)|: v\in \overline{B_1(0)},\ 1\leq i,j \leq d \}.
	\end{align*}
	If $\Sigma$ additionally satisfies \hyperlink{ass:coeff-growth-one}{($\Sigma$3)}, then we define
	\[
		N_\Sigma\defeq \sqrt{M_\Sigma^2+(B_\Sigma\vee M)^2}.
	\]
	If instead \hyperlink{ass:coeff-growth-two}{($\Sigma$3$'$)} is fulfilled, then we consider instead
	\[
		N_\Sigma\defeq \sqrt{M_\Sigma^2+B_\Sigma^2+dM^2}.
	\]
\end{definition}
\begin{definition}\label{def:operators}
	Let $D=C_c^\infty(E)$ be the space of compactly supported smooth functions on $E$. We define the linear operators $S,A$ and $L$ on $D$ via
	\begin{align*}
		Sf &= \sum_{i,j=1}^d a_{ij}\partial_{v_j}\partial_{v_i}f + \sum_{i=1}^d b_i\partial_{v_i}f, \\
		   &\quad\text{ where } b_i(v)= \sum_{j=1}^d(\partial_j a_{ij}(v)-a_{ij}(v)v_j),\\
		Af &= \nabla\Phi(x)\cdot\nabla_v f - v\cdot\nabla_x f, \\
		Lf &= (S-A)f, \qquad \text{ for }f\in D.
	\end{align*}
	
	Integration by parts shows that $(S,D)$ is symmetric and non-positive definite on $H$, and $(A,D)$ is antisymmetric on $H$.
	Hence, all three operators with domain $D$ are dissipative and therefore closable.
	We denote their closure respectively by $(S,D(S)), (A,D(A))$ and $(L,D(L))$.
\end{definition}
For $f\in D$ and $g\in H^{1,2}(E,\mu)$, integration by parts yields
\[
(Lf,g)_H = -\int_E \left(\nabla f, \begin{pmatrix}0 & -I\\I & \Sigma\end{pmatrix}\nabla g \right)_\mathrm{euc}\, \mathrm{d}\mu.
\]
In particular, (D6) is obviously fulfilled.
Next we provide an estimate which will be needed later:
\begin{proposition}\label{prop:b-bound}
	Let \nref{ass:coeff-derivatives}{($\Sigma$2)}  and either \nref{ass:coeff-growth-one}{($\Sigma$3)} or \nref{ass:coeff-growth-two}{($\Sigma$3$'$)} hold respectively and recall Definition~\ref{def:coeff-const}.
	Then for all $1\leq i,j\leq d$, it holds that
	\[
		\|\partial_j a_{ij}-a_{ij}v_j\|_{L^2(\nu)}\leq N_\Sigma.
	\]
\end{proposition}
\begin{proof}
	Due to integration by parts, it holds that
	\[
		\int_{\R^d}a_{ij}^2v_j^2\,\mathrm{d}\nu = \int_{\R^d} a_{ij}^2 + 2 a_{ij}v_j\partial_j a_{ij}\,\mathrm{d}\nu.
	\]
	Hence we obtain in the case \nref{ass:coeff-growth-two}{($\Sigma$3$'$)}
	\[
	\begin{aligned}
		\int_{\R^d} (\partial_j a_{ij}-a_{ij}v_j)^2\,\mathrm{d}\nu
		&= \int_{\R^d} (\partial_j a_{ij})^2 + a_{ij}^2\,\mathrm{d}\nu\\
		&\leq \int_{B_1(0)} (\partial_j a_{ij})^2\,\mathrm{d}\nu
			+ \int_{\R^d\setminus B_1(0)} (\partial_j a_{ij})^2\,\mathrm{d}\nu
			+ M_\Sigma^2\\
		&\leq B_\Sigma^2 + \int_{\R^d\setminus B_1(0)} (M|v|^\beta)^2\,\mathrm{d}\nu + M_\Sigma^2 \\
		&\leq B_\Sigma^2 + M_\Sigma^2 + \sum_{k=1}^d M^2\int_{\R^d} v_k^2\,\mathrm{d}\nu
		= B_\Sigma^2 + M_\Sigma^2 + M^2d.
	\end{aligned}
	\]
	The case \nref{ass:coeff-growth-one}{($\Sigma$3)} follows from $(\partial_j a_{ij})^2\leq (B_\Sigma\vee M)^2$.
\end{proof}

We now state the essential m-dissipativity result, which will be proven in the next section.
\begin{theorem}\label{thm:ess-m-diss}
	Let \nref{ass:ellipticity}{($\Sigma$1)}, \nref{ass:coeff-derivatives}{($\Sigma$2)} and either \nref{ass:coeff-growth-one}{($\Sigma$3)} or \nref{ass:coeff-growth-two}{($\Sigma$3$'$)} be fulfilled, and let $\Phi$ be as in \nref{ass:potential}{(P)}.
	Assume further that $\Phi$ is bounded from below and that $|\nabla\Phi|\in L^2(\R^d,\mathrm{e}^{-\Phi(x)}\,\mathrm{d}x)$.
	If $\beta$ is larger than $-1$,
	then assume additionally that there is some $N<\infty$ such that
	\[
		|\nabla\Phi(x)|\leq N(1+|x|^\gamma),\quad\text{ where }\gamma<\frac{2}{1+\beta}.
	\]
	Then the linear operator $(L,C_c^\infty(\R^{2d}))$ is essentially m-dissipative
	and hence its closure $(L,D(L))$ generates a strongly continuous contraction semigroup on $H$.
	In particular, the conditions (D2)-(D4) are satisfied.
\end{theorem}
Let us now introduce the orthogonal projections $P_S$ and $P$:
\begin{definition}
	Define $P_S:H\to H$ as
	\[
		P_Sf = \int_{\R^d} f \,\mathrm{d}\nu(v),\qquad f\in H,
	\]
	where integration is understood w.r.t~the velocity variable $v$.
	By Fubini's theorem and the fact that $\nu$ is a probability measure on $(E,\mathcal{F})$, it follows that $P_S$ is a well-defined orthogonal projection on $H$ with
	\[
		P_Sf\in L^2(\R^d,\mathrm{e}^{-\Phi(x)}\,\mathrm{d}x),\quad
		\|P_Sf\|_{L^2(\R^d,\mathrm{e}^{-\Phi(x)}\,\mathrm{d}x)} = \|P_Sf\|_H,\quad
		f\in H,
	\]
	where $L^2(\R^d,\mathrm{e}^{-\Phi(x)}\,\mathrm{d}x)$ is interpreted as embedded in $H$.
	
	Then define $P:H\to H$ via $Pf= P_Sf-(f,1)_H$ for $f\in H$. Again, $P$ is an orthogonal projection on $H$ with
	\[
		Pf\in L^2(\R^d,\mathrm{e}^{-\Phi(x)}\,\mathrm{d}x),\quad
		\|Pf\|_{L^2(\R^d,\mathrm{e}^{-\Phi(x)}\,\mathrm{d}x)} = \|Pf\|_H,\quad
		f\in H.
	\]
	Additionally, for each $f\in D$, $P_S f$ admits a unique representation in $C_c^\infty (\R^d)$,
	which we will denote by $f_S\in C_c^\infty (\R^d)$.
\end{definition}
In order to show the last remaining conditions (D5) and (D7),
we will make use of a standard sequence of cutoff functions as specified below:
\begin{definition}\label{def:cutoff}
	Let $\varphi\in C_c^\infty(\R^d)$ such that $0\leq\varphi\leq1$, $\varphi=1$ on $B_1(0)$ and $\varphi=0$ outside of $B_2(0)$. Define $\varphi_n(z) \defeq \varphi(\frac{z}{n})$ for each $z\in\R^d$, $n\in\N$.
	Then there exists a constant $C<\infty$ independent of $n\in\N$ such that
	\[
		|\partial_i\varphi_n(z)|\leq \frac{C}{n},\ 
		|\partial_{ij}\varphi_n(z)| \leq \frac{C}{n^2}\quad
		\text{ for all } z\in\R^d, 1\leq i,j \leq d.
	\]
	Moreover $0\leq\varphi_n\leq 1$ for all $n\in\N$
	and $\varphi_n\to 1$ pointwisely on $\R^d$ as $n\to\infty$.
\end{definition}
\begin{lemma}
	Let \nref{ass:coeff-derivatives}{($\Sigma$2)} and either \nref{ass:coeff-growth-one}{($\Sigma$3)} or \nref{ass:coeff-growth-two}{($\Sigma$3$'$)} be fulfilled, and let $\Phi$ be as in \nref{ass:potential}{(P)}. Then the operator $L$ satisfies the following:
	\begin{enumerate}[(\roman*)]
		\item $P(H)\subset D(S)$ with $SPf=0$ for all $f\in H$,
		\item $P(D)\subset D(A)$ and $APf= -v\cdot\nabla_x(P_Sf)$,
		\item $AP(D)\subset D(A)$ with $A^2Pf = \langle v,\nabla_x^2(P_Sf)v \rangle - \nabla\Phi\cdot\nabla_x(P_Sf)$.
		\item It holds $1\in D(L)$ and $L1=0$.
	\end{enumerate}
	In particular, (D5) and (D7) are fulfilled.
\end{lemma}
\begin{proof}
	We only show (i), as the other parts can be shown exactly as in \cite{GS16}.
	First, let $f\in C_c^\infty(\R^d)$ and define $f_n\in D$ via $f_n(x,y)\defeq f(x)\varphi_n(v)$.
	Then by Lebesgue's dominated convergence theorem and the inequalities in the previous definition,
	\[
		Sf_n= f\cdot\left(\sum_{i,j=1}^d a_{ij}\partial_{ij}\varphi_n + \sum_{i=1}^d b_i\partial_i\varphi_n\right) \to 0\quad\text{ in $H$ as }n\to\infty,
	\]
	since $a_{ij}\in L^\infty(\R^d)\subset L^2(\R^d,\nu)$, $|v|\in L^2(\R^d,\nu)$ and
	$\partial_j a_{ij}\in L^2(\R^d,\nu)$ for all $1\leq i,j\leq d$.
	
	Since $f_n\to f$ in $H$ and by closedness of $(S,D(S))$, this implies
	$f\in D(S)$ with $Sf=0$, where $f$ is interpreted as an element of $H$.
	
	Now let $g\in P(H)$ and identify $g$ as an element of $L^2(\R^d,\mathrm{e}^{-\Phi(x)}\,\mathrm{d}x)$.
	Then there exist $g_n\in C_c^\infty(\R^d)$ with
	$g_n\to g$ in $L^2(\R^d,\mathrm{e}^{-\Phi(x)}\,\mathrm{d}x)$ as $n\to\infty$.
	Identifying all $g_n$ and $g$ with elements in $H$ then yields
	$g_n\to g$ in $H$ as $n\to\infty$ and $g_n\in D(S)$, $Sg_n=0$ for all $n\in \N$.
	Therefore, again by closedness of $(S,D(S))$, $g\in D(S)$ and $Sg=0$.
\end{proof}

\subsection{The hypocoercivity conditions}
Now we verify the hypocoercivity conditions \nref{ass:algebraic}{(H1)}-\nref{ass:aux-bound}{(H4)} for the operator $L$.
From here on, we will assume $\Sigma$ to satisfy \nref{ass:ellipticity}{($\Sigma$1)}, \nref{ass:coeff-derivatives}{($\Sigma$2)} and either \nref{ass:coeff-growth-one}{($\Sigma$3)} or \nref{ass:coeff-growth-two}{($\Sigma$3$'$)}, with $N_\Sigma$ referring to the appropriate constant
as in Definition~\ref{def:coeff-const}.
Analogously to \cite{GS16} we introduce the following conditions:
\begin{hypoc_potential}[\hypertarget{ass:hypoc-potential}{C1)-(C3}]We require the following assumptions on $\Phi:\R^d\to\R$:
	\begin{enumerate}[(C1)]
		\item The potential $\Phi$ is bounded from below, is an element of $C^2(\R^d)$ and
			$\mathrm{e}^{-\Phi(x)}\,\mathrm{d}x$ is a probability measure on $(\R^d,\mathcal{B}(\R^d))$.
		\item The probability measure $\mathrm{e}^{-\Phi(x)}\,\mathrm{d}x$ satisfies a Poincaré inequality
			of the form
			\[
				\|\nabla f\|_{L^2(\mathrm{e}^{-\Phi(x)}\,\mathrm{d}x)}^2\geq
				\Lambda \|f-(f,1)_{L^2(\mathrm{e}^{-\Phi(x)}\,\mathrm{d}x)}\|_{L^2(\mathrm{e}^{-\Phi(x)}\,\mathrm{d}x)}^2
			\]
			for some $\Lambda\in (0,\infty)$ and all $f\in C_c^\infty(\R^d)$.
		\item There exists a constant $c<\infty$ such that
			\[
				|\nabla^2\Phi(x)| \leq c(1+|\nabla\Phi(x)|)
				\quad\text{ for all }x\in\R^d.
			\]
	\end{enumerate}
	Note that in particular, (C1) implies \nref{ass:potential}{(P)}.
	As shown in \cite[Lemma A.24]{Villani}, conditions (C3) and (C1) imply $\nabla\Phi\in L^2(\mathrm{e}^{-\Phi(x)}\,\mathrm{d}x)$.
\end{hypoc_potential}
Since we only change the operator $(S,D(S))$ in comparison to the framework of \cite{GS16},
the results stated there involving only $(A,D(A))$ and the projections also hold here and are collected as follows:
\begin{proposition} Let $\Phi$ satisfy \nref{ass:potential}{(P)}. Then the following hold:
	\begin{enumerate}[(i)]
		\item Assume additionally $\nabla\Phi\in L^2(\mathrm{e}^{-\Phi(x)}\,\mathrm{d}x)$. Then \nref{ass:algebraic}{(H1)} is fulfilled.
		\item Assume that $\Phi$ satisfies (C1) and that $\nabla\Phi\in L^2(\mathrm{e}^{-\Phi(x)}\,\mathrm{d}x)$.
			Then the operator $(G,D)$ defined by $G\defeq PA^2P$ is essentially self-adjoint,
			equivalently essentially m-dissipative. For $f\in D$, it holds
			\[
				Gf = PAAPf = \Delta f_S - \nabla\Phi\cdot \nabla f_S.
			\]
		\item Assume that $\Phi$ satisfies (C1) and (C2) as well as $\nabla\Phi\in L^2(\mathrm{e}^{-\Phi(x)}\,\mathrm{d}x)$.
			Then \nref{ass:macro-coerc}{(H3)} holds with $\Lambda_M = \Lambda$.
		\item Assume that $\Phi$ satisfies (C1)-(C3). Then the second estimate in \nref{ass:aux-bound}{(H4)} is satisfied,
			and the constant there is given as $c_2 = c_\Phi\in [0,\infty)$, which only depends on the choice of $\Phi$.
	\end{enumerate}
\end{proposition}
It remains to show \nref{ass:micro-coerc}{(H2)} and the first half of \nref{ass:aux-bound}{(H4)}:
\begin{proposition}
	Let $\Phi$ be as in \nref{ass:potential}{(P)}.
	Then Condition \nref{ass:micro-coerc}{(H2)} is satisfied with $\Lambda_m=c_\Sigma$.
\end{proposition}
\begin{proof}
	Let $g\in C_c^\infty(\R^d)$. The Poincaré inequality for Gaussian measures, see for example \cite{GaussPoincare}, states
	\[
		\|\nabla g\|_{L^2(\nu)}^2 \geq \left\| g-\int_{\R^d} g(v)\,\mathrm{d}\nu(v) \right\|_{L^2(\nu)}^2.
	\]
	Therefore, integration by parts yields for all $f\in D$:
	\[
	\begin{aligned}
		(-Sf,f)_H &= \int_E \langle \nabla_v f, \Sigma\nabla_v f\rangle \,\mathrm{d}\mu
			\geq \int_{\R^d} \int_{\R^d} c_\Sigma |\nabla_v f(x,v)|^2\,\mathrm{d}\nu\,\mathrm{e}^{-\Phi(x)}\,\mathrm{d}x\\
			&\geq c_\Sigma \int_{\R^d} \int_{\R^d} (f-P_Sf)^2\,\mathrm{d}\nu\,\mathrm{e}^{-\Phi(x)}\,\mathrm{d}x
			= c_\Sigma \|(I-P_S)f\|_H^2
	\end{aligned}
	\]
\end{proof}
Finally, we verify the first part of \nref{ass:aux-bound}{(H4)}:
\begin{proposition}\label{prop:aux-bound-sym}
	Assume that $\Phi$ satisfies (C1) and (C2) as well as $\nabla\Phi\in L^2(\mathrm{e}^{-\Phi}\,\mathrm{d}x)$.
	Then the first inequality of \nref{ass:aux-bound}{(H4)} is also satisfied with $c_1=d_\Sigma\defeq\sqrt{2d^3}N_\Sigma$.
\end{proposition}
\begin{proof}
	For $f\in D$, define $Tf\in H$ by 
	\[
		Tf\defeq \sum_{i=1}^d b_i\partial_i(f_S) = \sum_{i,j=1}^d (\partial_j a_{ij} - a_{ij}v_j)\partial_{x_i}(P_S f).
	\]
	We want to apply Lemma~\ref{lem:auxiliary-bound} to the operator $(S,D(S))$.
	Let $f\in D$, $h\in D(S)$ and $h_n\in D$ such that $h_n\to h$ and $Sh_n\to Sh$ in $H$ as $n\to\infty$.
	Then, by integration by parts,
	\[
		(Sh,APf)_H= \lim_{n\to\infty}(Sh_n,-v\cdot\nabla_x(P_S f))_H = \lim_{n\to\infty} (h_n,-Tf)_H= (h,-Tf)_H.
	\]
	This shows $APf\in D(S^*)$ and by the first part of Lemma~\ref{lem:auxiliary-bound},
	$(I-G)f\in D((BS)^*)$ and $(BS)^*(I-G)f= S^*APf = -Tf$. Now set $g=(I-G)f$, then, via Proposition~\ref{prop:b-bound},
	\[
	\begin{aligned}
		\|(BS)^*g\|_H^2&=\|Tf\|_H^2=\int_E \left(\sum_{i=1}^d b_i  \partial_i f_S \right)^2\, \mathrm{d}\mu\\
		&\leq d^2 \sum_{i,j=1}^d\int_{\R^d} \int_{\R^d}(\partial_j a_{ij}(v)-a_{ij}(v)v_j)^2\ \mathrm{d}\nu(v)\ (\partial_{x_i}(P_Sf)(x))^2 \,\mathrm{e}^{-\Phi(x)}\ \mathrm{d}x\\
		&\leq d^3N_\Sigma^2 \sum_{i=1}^d\int_{\R^d} \partial_{x_i}(Pf)\cdot \partial_{x_i}(P_Sf)\,\mathrm{e}^{-\Phi(x)}\ \mathrm{d}x.
	\end{aligned}
	\]
	A final integration by parts then yields
	\[
	\begin{aligned}
		\|(BS)^*g\|_H^2
		&\leq -d^3N_\Sigma^2 \int_{\R^d}Pf\cdot
			(\Delta_x P_Sf - \nabla\Phi\nabla_x(P_Sf))\,\mathrm{e}^{-\Phi(x)}\ \mathrm{d}x\\
		&= -d^3N_\Sigma^2 \int_{\R^d}Pf\cdot
		Gf\,\mathrm{e}^{-\Phi(x)}\ \mathrm{d}x\\
		&\leq d^3N_\Sigma^2\, \|Pf\|_{L^2(\text{e}^{-\Phi(x)}\,\mathrm{d}x)}\cdot\|Gf\|_{L^2(\text{e}^{-\Phi(x)}\,\mathrm{d}x)}\\
		&\leq d^3N_\Sigma^2\, \|Pf\|_H(\|(I-G)f\|_H + \|f\|_H)\\
		&\leq 2d^3N_\Sigma^2\, \|g\|_H^2,
	\end{aligned}
	\]
	where the last inequality is due to dissipativity of $(G,D)$.
\end{proof}
\begin{proof}[of Theorem 1.1]
	Under the given assumptions, all conditions \nref{ass:hypoc-potential}{(C1)-(C3)}, \nref{ass:ellipticity}{($\Sigma$1)}, \nref{ass:coeff-derivatives}{($\Sigma$2)} and either \nref{ass:coeff-growth-one}{($\Sigma$3)} or \nref{ass:coeff-growth-two}{($\Sigma$3$'$)} are satisfied.
	Therefore hypocoercivity follows by the previous propositions and Theorem~\ref{thm:hypoc}.
	It remains to show the stated convergence rate, which will be done as in \cite{GS16} or \cite{Hypo2D}
	using the determined values for $c_1$, $c_2$, $\Lambda_M$ and $\Lambda_m$.
	Fix
	\[
		\delta\defeq \frac{\Lambda}{1+\Lambda}\frac1{1+c_\Phi+d_\Sigma}.
	\]
	Then the coefficients on the right hand side of \eqref{eq:conv-rate} can be written as
	$c_\Sigma-\eps r_{\Phi}(N_\Sigma)$ and $\eps s_\Phi$ respectively, where
	\[
	\begin{aligned}
		r_\Phi(N_\Sigma)&\defeq (1+c_\Phi+\sqrt{2d^3}N_\Sigma)\left(1+\frac{1+\Lambda}{2\Lambda}(1+c_\Phi+\sqrt{2d^3}N_\Sigma) \right) \quad\text{ and }\\
		s_\Phi &\defeq \frac12\frac\Lambda{1+\Lambda}.
	\end{aligned}
	\]
	and $\eps=\eps_\Phi(\Sigma)\in(0,1)$ still needs to be determined.
	Write $r_{\Phi}(N_\Sigma)+s_\Phi$ as the polynomial
	\[
		r_{\Phi}(N_\Sigma)+s_\Phi = a_1+a_2N_\Sigma+a_3N_\Sigma^2,
	\]
	where all $a_i\in(0,\infty)$, $i=1,\dots,3$ depend on $\Phi$.
	Then define
	\[
		\tilde{\eps}_\Phi(N_\Sigma)\defeq \frac{N_\Sigma}{r_\Phi(N_\Sigma)+s_\Phi}=\frac{N_\Sigma}{a_1+a_2N_\Sigma+a_3N_\Sigma^2}.
	\]
	Some rough estimates show $\tilde{\eps}_\Phi(N_\Sigma)\in(0,1)$.
	Now let $v>0$ be arbitrary and set
	\[
		\eps\defeq \frac{v}{1+v}\frac{c_\Sigma}{N_\Sigma}\tilde{\eps}_\Phi(N_\Sigma)\in(0,1).
	\]
	Then $\eps r_\Phi(N_\Sigma)+\eps s_\Phi=\frac{v}{1+v} c_\Sigma <c_\Sigma$,
	hence we get the estimate
	\[
		c_\Sigma-\eps r_\Phi(N_\Sigma) > \eps s_\Phi = \frac{v}{1+v}\frac{2c_\Sigma}{n_1+n_2N_\Sigma+n_3N_\Sigma^2}\eqdef\kappa,
	\]
	where all $n_i\in(0,\infty)$ depend on $\Phi$ and are given by
	\[
		n_i \defeq \frac2{s_\Phi}a_i,\qquad\text{ for each }i=1,\dots,3.
	\]
	Clearly, $\kappa$, $\eps$ and $\delta$ now solve \eqref{eq:conv-rate} and the convergence rate coefficients are given via
	Theorem~\ref{thm:hypoc} by
	\begin{align*}
		\kappa_1&=\sqrt{\frac{1+\eps}{1-\eps}}
		=\sqrt{\frac{1+v+\frac{c_\Sigma}{N_\Sigma}\tilde{\eps}_\Phi(N_\Sigma)v}
			{1+v-\frac{c_\Sigma}{N_\Sigma}\tilde{\eps}_\Phi(N_\Sigma)v}}
		\leq \sqrt{1+2v+v^2}=1+v \quad\text{ and }\\
		\kappa_2 &= \frac\kappa{1+\eps} > \frac12\kappa
	\end{align*}
	Hence, by choosing $\theta_1=1+v$ and $\theta_2= \frac12\kappa= \frac{\theta_1-1}{\theta_1}\frac{c_\Sigma}{n_1+n_2N_\Sigma+n_3N_\Sigma^2}$, the rate of convergence
	claimed in the theorem is shown.
\end{proof}
\begin{remark}\label{rem:magic}
	We remark here that all previous considerations up to the explicit rate of convergence can also be applied
	to the formal adjoint operator $(L^*,D)$ with $L^*=S+A$, the closure of which generates the adjoint semigroup $(T_t^*)_{t\geq 0}$ on $H$. For example, the perturbation procedure to prove essential m-dissipativity is exactly the same as for $L$, since the sign of $A$ does not matter due to antisymmetry. We can use this to construct solutions to the corresponding Fokker-Planck PDE associated with our Langevin dynamics, see Section~\ref{subsec:fokker-planck}.
\end{remark}

\section{Essential m-dissipativity of the Langevin operator}
The goal of this section is to prove Theorem~\ref{thm:ess-m-diss}.
We start by giving some basics on perturbation of semigroup generators.
\subsection{Basics on generators and perturbation}
\begin{definition}
	Let $(A,D(A))$ and $(B,D(B))$ be linear operators on $H$.
	Then $B$ is said to be \emph{$A$-bounded} if $D(A)\subset D(B)$
	and there exist constants $a,b<\infty$ such that
	\begin{equation}\label{eq:a-bound}
		\|Bf\|_H\leq a\|Af\|_H + b\|f\|_H
	\end{equation}
	holds for all $f\in D(A)$.
	The number $\inf\{ a\in\R\mid \text{ \eqref{eq:a-bound} holds for some }b<\infty \}$
	is called the \emph{$A$-bound} of $B$.
\end{definition}
\begin{theorem}
	Let $D\subset H$ be a dense linear subspace, $(A,D)$ be an essentially m-dissipative linear operator on $H$
	and let $(B,D)$ be dissipative and $A$-bounded with $A$-bound strictly less than $1$.
	Then $(A+B,D)$ is essentially m-dissipative and its closure is given by
	$(\overline{A}+\overline{B},D(\overline{A}))$.
\end{theorem}
A useful criterion for verifying $A$-boundedness is given by:
\begin{lemma}\label{lem:perturbation-criterion}
	Let $D\subset H$ be a dense linear subspace, $(A,D)$ be essentially m-dissipative
	and $(B,D)$ be dissipative. Assume that there exist constants $c,d<\infty$ such that
	\[
		\|Bf\|_H^2 \leq c(Af,f)_H + d\|f\|_H^2
	\]
	holds for all $f\in D$. Then $B$ is $A$-bounded with $A$-bound $0$.
\end{lemma}
We also require the following generalization of the perturbation method:
\begin{lemma}\label{lem:perturbation-projection}
	Let $D\subset H$ be a dense linear subspace, $(A,D)$ be essentially m-dissipative
	and $(B,D)$ be dissipative on $H$. Assume that there exists a complete orthogonal family
	$(P_n)_{n\in\N}$, i.e.~each $P_n$ is an orthogonal projection, $P_nP_m=0$ for all $n\neq m$
	and $\sum_{n\in\N}P_n= I$ strongly, such that
	\[
		P_n(D)\subset D,\qquad P_nA=AP_n,\quad\text{ and }\quad P_nB=BP_n
	\]
	for all $n\in\N$. Set $A_n\defeq AP_n$, $B_n\defeq BP_n$, both with domain $D_n\defeq P_n(D)$,
	as operators on $P_n(H)$. Assume that each $B_n$ is $A_n$-bounded with $A_n$-bound strictly less than $1$.
	Then $(A+B,D)$ is essentially m-dissipative.
\end{lemma}

\subsection{The symmetric part}
We first prove essential self-adjointness, equivalently essential m-dissipativity, for a certain class of symmetric differential operators on specific Hilbert spaces. This is essentially a combination of two results by Bogachev, Krylov, and Röckner,
namely \cite[Corollary 2.10]{BKR01} and \cite[Theorem 7]{BKR97}, however, the combined statement does not seem to be well known and might hold interest as the basis for similar m-dissipativity proofs.
We use the slightly more general statement from \cite[Theorem 5.1]{BGS13} in order to relax the assumptions.
\begin{theorem}\label{thm:ess-self-adjoint}
	Let $d\geq 2$ and consider $H=L^2(\R^d, \mu)$ where $\mu=\rho\,\mathrm{d}x$, $\rho=\varphi^2$ for some
	$\varphi\in H_{\mathrm{loc}}^{1,2}(\R^d)$ such that $\frac1\rho\in L_{\mathrm{loc}}^\infty(\R^d)$.
	Let $A=(a_{ij})_{1\leq i,j\leq d}:\R^d\to \R^{d\times d}$
	be symmetric and locally strictly elliptic with $a_{ij}\in L^\infty(\R^d)$ for all $1\leq i,j\leq d$.
	Assume there is some $p>d$ such that $a_{ij}\in H_{\mathrm{loc}}^{1,p}(\R^d)$ for all $1\leq i,j\leq d$
	and that $|\nabla\rho|\in L_{\mathrm{loc}}^p(\R^d)$. Consider the bilinear form $(B,D)$ given by $D=C_c^\infty(\R^d)$
	and
	\[
		B(f,g)\defeq (\nabla f,A\nabla g)_H= \int_{\R^d} (\nabla f(x),A(x)\nabla g(x))_{\mathrm{euc}}\,\rho(x)\,\mathrm{d}x,
		\qquad f,g\in D.
	\]
	Define further the linear operator $(S,D)$ via
	\[
		Sf\defeq \sum_{i,j=1}^d a_{ij}\partial_j\partial_i f + \sum_{i=1}^d b_i\partial_i f,\qquad f\in D,
	\]
	where $b_i=\sum_{j=1}^d (\partial_j a_{ij}+a_{ij}\frac{\partial_j\rho}{\rho})\in L_{\mathrm{loc}}^p(\R^d)$, so that
	$B(f,g)=(-Sf,g)_H$.
	Then $(S,D)$ is essentially self-adjoint on $H$.
\end{theorem}
\begin{proof}
	Analogously to the proof of \cite[Theorem 7]{BKR97}, it can be shown that $\rho$ is continuous, hence locally bounded.
	Assume that there is some $g\in H$ such that
	\begin{equation}\label{eq:range-orthogonal}
		\int_{\R^d} (S-I)f(x)\cdot g(x)\cdot \rho(x)\,\mathrm{d}x = 0\quad\text{ for all }f\in D.
	\end{equation}
	Define the locally finite signed Borel measure $\nu$ via $\nu=g\rho\,\mathrm{d}x$, which is then absolutely continuous
	with respect to the Lebesgue measure. By definition it holds that
	\[
		\int_{\R^d} \left( \sum_{i,j=1}^d a_{ij}\partial_j\partial_if + \sum_{i=1}^d b_i\partial_i f - f \right)\,\mathrm{d}\nu =0
		\quad\text{ for all }f\in D,
	\]
	so by \cite[Theorem 5.1]{BGS13}, the density $g\cdot\rho$ of $\nu$ is in $H_{\mathrm{loc}}^{1,p}(\R^d)$ and locally Hölder continuous, hence locally bounded. This implies $g=g\rho\cdot\frac1\rho\in L_{\mathrm{loc}}^p(\R^d)\cap L_{\mathrm{loc}}^\infty(\R^d)$ and 
	$\nabla g=\nabla(g\rho)\cdot\frac1\rho-(g\rho)\frac{\nabla \rho}{\rho^2}\in L_{\mathrm{loc}}^p(\R^d)$.
	Hence $g\in H_{\mathrm{loc}}^{1,p}(\R^d)$, is locally bounded, and $g\cdot b_i\in L_{\mathrm{loc}}^p(\R^d)$ for all $1\leq i\leq d$. Therefore, we can apply integration by parts to \eqref{eq:range-orthogonal} and get for every $f\in D$:
	\begin{equation}\label{eq:gradient-form-zero}
	\begin{aligned}
		0 &= -\sum_{i,j=1}^d (a_{ij}\partial_if,\partial_jg)_H -\sum_{i=1}^d (\partial_if, b_ig)_H + \sum_{i=1}^d (\partial_if, b_ig)_H - (f,g)_H\\
		 &= -\int_{\R^d} (\nabla f,A\nabla g)_\mathrm{euc}\,\mathrm{d}\mu-(f,g)_H.
	\end{aligned}
	\end{equation}
	Note that this equation can then be extended to all $f\in H^{1,2}(\R^d)$ with compact support, since $p>2$ by definition.
	Now let $\psi\in C_c^\infty(\R^d)$ and set $\eta=\psi g\in H^{1,2}(\R^d)$, which has compact support.
	The same then holds for $f\defeq \psi \eta\in H^{1,2}(\R^d)$. Elementary application of the product rule yields
	\begin{equation}\label{eq:product-rule}
		(\nabla \eta, A\nabla (\psi g))_\mathrm{euc} = (\nabla f,A\nabla g)_\mathrm{euc} - \eta(\nabla\psi,A\nabla g)_\mathrm{euc}
		+ g(\nabla\eta,A\nabla\psi)_\mathrm{euc}.
	\end{equation}
	From now on, for $a,b:\R^d\to\R^d$, let $(a,b)$ always denote the evaluation of the Euclidean inner product $(a,b)_\mathrm{euc}$. By using \eqref{eq:product-rule} and applying \eqref{eq:gradient-form-zero} to $f$, we get
	\[
	\begin{aligned}
		\int_{\R^d} &(\nabla(\psi g),A\nabla(\psi g))\,\mathrm{d}\mu + \int_{\R^d} (\psi g)^2\,\mathrm{d}\mu
		= \int_{\R^d} (\nabla\eta,A\nabla(\psi g))\,\mathrm{d}\mu + \int_{\R^d} \eta\psi g\,\mathrm{d}\mu\\
		&= \int_{\R^d} (\nabla f,A\nabla g)\,\mathrm{d}\mu
			- \int_{\R^d} \eta(\nabla\psi,A\nabla g)\,\mathrm{d}\mu
			+ \int_{\R^d} g(\nabla\eta,A\nabla\psi)\,\mathrm{d}\mu
			+ \int_{\R^d} f g\,\mathrm{d}\mu\\
		&= - \int_{\R^d} \psi g(\nabla\psi,A\nabla g)\,\mathrm{d}\mu
			+ \int_{\R^d} g(\nabla(\psi g),A\nabla\psi)\,\mathrm{d}\mu\\
		&= \int_{\R^d} g^2(\nabla\psi,A\nabla\psi)\,\mathrm{d}\mu,
	\end{aligned}
	\]
	where the last step follows from the product rule and symmetry of $A$.
	Since $A$ is locally strictly elliptic, there is some $c>0$ such that
	\[
		0\leq \int_{\R^d} c(\nabla(\psi g),\nabla(\psi g)) \,\mathrm{d}\mu
			\leq \int_{\R^d} (\nabla(\psi g),A\nabla(\psi g))\,\mathrm{d}\mu
	\]
	and therefore it follows that
	\begin{equation}\label{eq:final-stretch}
		\int_{\R^d} (\psi g)^2\,\mathrm{d}\mu \leq \int_{\R^d} g^2(\nabla\psi,A\nabla\psi)\,\mathrm{d}\mu.
	\end{equation}
	Let $(\psi_n)_{n\in\N}$ be as in Definition~\ref{def:cutoff}. Then \eqref{eq:final-stretch} holds for all $\psi=\psi_n$.
	By dominated convergence, the left part converges to $\|g\|_H^2$ as $n\to\infty$. The integrand of the right hand side term
	is dominated by $d^2C^2M \cdot g^2\in L^1(\mu)$, where $C$ is from Def.~\ref{def:cutoff} and $M\defeq\max_{1\leq i,j\leq d}\|a_{ij}\|_\infty$.
	By definition of the $\psi_n$, that integrand converges pointwisely to zero as $n\to\infty$, so again by dominated convergence it follows that $g=0$ in $H$.
	
	This implies that $(S-I)(D)$ is dense in $H$ and therefore that $(S,D)$ is essentially self-adjoint.
\end{proof}
\begin{remark}
	The above theorem also holds for $d=1$, as long as $p\geq 2$. Indeed, continuity of $\rho$ follows from similar regularity estimates, see \cite[Remark 2]{BKR97}. The proof of \cite[Theorem 5.1]{BGS13} mirrors the proof of \cite[Theorem 2.8]{BKR01}, where $d\geq 2$ is used to apply \cite[Theorem 2.7]{BKR01}. However, in the cases where it is applied, this distinction is not necessary (since $p'<q$ always holds). Finally, the extension of \eqref{eq:gradient-form-zero} requires $p\geq 2$.
\end{remark}
We use this result to prove essential m-dissipativity of the symmetric part $(S,D)$ of our operator $L$:
\begin{theorem}
Let $H, D$ and the operator $S$ be defined as in Section~\ref{subsec:data}. Then $(S,D)$ is essentially m-dissipative on $H$.
Its closure $(S,D(S))$ generates a sub-Markovian strongly continuous contraction semigroup on $H$.
\end{theorem}
\begin{proof}
	Define the operator $(\tilde{S},C_c^\infty(\R^d))$ on $L^2(\R^d,\nu)$ by
	\[
		\tilde{S}f\defeq \sum_{i,j=1}^d a_{ij}\partial_j\partial_if + \sum_{i=1}^d b_i\partial_if,
		\quad f\in C_c^\infty(\R^d).
	\]
	The density $\rho$ of $\nu$ wrt.~the Lebesgue measure is given by $\rho(v)=\mathrm{e}^{-v^2/2}=(\mathrm{e}^{-v^2/4})^2$.
	Due to the conditions \nref{ass:ellipticity}{($\Sigma$1)}, \nref{ass:coeff-derivatives}{($\Sigma$2)} and either \nref{ass:coeff-growth-one}{($\Sigma$3)} or \nref{ass:coeff-growth-two}{($\Sigma$3$'$)}, all assumptions from Theorem~\ref{thm:ess-self-adjoint} are fulfilled and therefore,
	$(\tilde{S},C_c^\infty(\R^d))$ is essentially m-dissipative in $L^2(\nu)$. Let $g=g_1\otimes g_2\in C_c^\infty(\R^d)\otimes C_c^\infty(\R^d)$ be a pure tensor. Then there is a sequence $(\tilde{f}_n)_{n\in\N}$ in $C_c^\infty(\R^d)$ such that $(I-\tilde{S})\tilde{f}_n\to g_2$ in $L^2(\nu)$ as $n\to\infty$. Define $f_n \in D$ for each $n\in\N$ by
	\[
		f_n(x,v) \defeq g_1(x)\tilde{f}_n(v).
	\]
	Then
	\[
		\|(I-S)f_n-g\|_H = \| g_1\otimes ((I-\tilde{S})\tilde{f}_n - g_2)\|_H= \|g_1\|_{L^2(\mathrm{e}^{-\Phi(x)}\,\mathrm{d}x)}\cdot \|(I-\tilde{S})\tilde{f}_n-g_2 \|_{L^2(\nu)},
	\]
	which converges to zero as $n\to\infty$. By taking linear combinations, this shows that $(I-S)(D)$ is dense in $C_c^\infty(\R^d)\otimes C_c^\infty(\R^d)$ wrt.~the $H$-norm. Since $C_c^\infty(\R^d)\otimes C_c^\infty(\R^d)$ is dense in $H$, $(S,D)$ is essentially m-dissipative and its closure $(S,D(S))$ generates a strongly continuous contraction semigroup.
	
	It can easily be shown that $(Sf,f^+)_H\leq 0$ for all $f\in D$.
	Parallelly to the proof of (D7), it holds that $1\in D(S)$ and $S1=0$. This together implies that $(S,D(S))$ is a Dirichlet operator
	and the generated semigroup is sub-Markovian.
\end{proof}

\subsection{Perturbation of the symmetric part for nice coefficients}
Now we extend the essential m-dissipativity stepwise to the non-symmetric operator $L$ by perturbation.
This follows and is mostly based on the method seen in the proof of \cite[Theorem 6.3.1]{Conrad2011}, which proved that result for $\Sigma=I$.

Since $S$ is dissipative on $D_1\defeq L_0^2(\mathrm{e}^{-\Phi}\,\mathrm{d}x)\otimes C_c^\infty(\R^d)\supset D$,
the operator $(S,D_1)$ is essentially m-dissipative as well.
The unitary transformation $T:L^2(\R^d,\mathrm{d}(x,v))\to H$ given by $Tf(x,v)=\mathrm{e}^{\frac{v^2}4+\frac{\Phi(x)}2}f(x,v)$ leaves $D_1$ invariant.
This implies that the operator $(S_1,D_1)$ on $L^2(\R^d,\mathrm{d}(x,v))$ , where $S_1=T^{-1}ST$, is again essentially m-dissipative. Note that $S_1$ is explicitly given by
\[
	S_1 f = \sum_{i,j=1}^d a_{ij}\partial_{v_j}\partial_{v_i}f- \frac14(v,\Sigma v)f+ \frac12 \operatorname{tr}(\Sigma)f+ \sum_{i,j=1}^d \partial_j a_{ij}(\frac{v_i}2f+\partial_{v_i}f)
\]
Now consider the operator $(ivxI, D_1)$, which is dissipative as $\operatorname{Re}(ivxf,f)_{L^2(\R^d,\mathrm{d}(x,v))}=0$ for $f\in D_1$.
We show the following perturbation result:
\begin{proposition}\label{prop:no-potential}
	Let $\Sigma$ satisfy \nref{ass:coeff-growth-one}{($\Sigma$3)} with $\beta\leq -1$. Then the operator
	$(S_1+ivxI,D_1)$ is essentially m-dissipative on $L^2(\R^d,\mathrm{d}(x,v))$.
\end{proposition}
\begin{proof}
Define the orthogonal projections $P_n$ via $P_nf(x,v)\defeq \xi_n(x)f(x,v)$,
where $\xi_n$ is given by $\xi_n = \mathds{1}_{[n-1,n)}(|x|)$,
which leave $D_1$ invariant. Then the conditions for Lemma~\ref{lem:perturbation-projection} are fulfilled,
and we are left to show the $A_n$-bounds.
Note that due to the restriction on $\beta$, there is some constant $C<\infty$ such that
$\partial_j a_{ij}(v)v_i\leq C$ for all $1\leq i,j\leq d$, $v\in\R^d$.
For each fixed $n\in\N$ it holds for all $f\in P_nD_1$:
\[
\begin{aligned}
	\|ivxf\|_{L^2}^2 &\leq n^2 \int_{\R^{2d}} |v|^2 f^2\,\mathrm{d}(x,v)
		\leq 4c_\Sigma^{-1} n^2 \int_{\R^{2d}}  \frac{(v,\Sigma v)}4 f^2\,\mathrm{d}(x,v)\\
		&\leq 4c_\Sigma^{-1} n^2 \int_{\R^{2d}} \frac{(v,\Sigma v)}4 f^2+(\nabla_v f,\Sigma \nabla_v f)  \,\mathrm{d}(x,v)\\
		&= 4c_\Sigma^{-1} n^2 \int_{\R^{2d}}\left(-\sum_{i,j=1}^d a_{ij}\partial_{v_j}\partial_{v_i}f-\sum_{i,j=1}^d\partial_j a_{ij}\partial_{v_i}f+\frac{(v,\Sigma v)}4 f  \right) f\,\mathrm{d}(x,v)\\
		&= 4c_\Sigma^{-1} n^2\left( (-P_nS_1f,f)
			+ \int_{\R^{2d}} \frac12 \operatorname{tr}(\Sigma)f^2
			+ \sum_{i,j=1}^d \partial_j a_{ij}\frac{v_i}2f^2 \,\mathrm{d}(x,v)\right)\\
		&\leq 4c_\Sigma^{-1} n^2\left( (-S_1f,f) + (d^2C+\frac{dM_\Sigma}2)\|f\|_{L^2}^2\right).
\end{aligned}
\]
Hence by Lemma~\ref{lem:perturbation-criterion}, $(ivxI P_n,P_n D_1)$ is $S_1 P_n$-bounded with Kato-bound zero.
Application of Lemma~\ref{lem:perturbation-projection} yields the statement.
\end{proof}
Since $C_c^\infty(\R^d)\otimes C_c^\infty(\R^d)$ is dense in $D_1$ wrt.~the graph norm of $S_1+ivxI$,
we obtain essential m-dissipativity of $(S_1+ivxI,C_c^\infty(\R^d)\otimes C_c^\infty(\R^d))$ and therefore also of its dissipative extension $(S_1+ivxI,D_2)$ with $D_2\defeq \mathcal{S}(\R^d)\otimes C_c^\infty(\R^d))$,
where $\mathcal{S}(\R^d)$ denotes the set of smooth functions of rapid decrease on $\R^d$.
Applying Fourier transform in the $x$-component leaves $D_2$ invariant and shows that
$(L_2,D_2)$ is essentially m-dissipative, where $L_2= S_1+v\nabla_x$.
Now we add the part depending on the potential $\Phi$.

\begin{proposition}
	Let $\Sigma$ satisfy \nref{ass:coeff-growth-one}{($\Sigma$3)} with $\beta\leq -1$ and $\Phi$ be Lipschitz-continuous.
	Then the operator $(L',D_2)$ with $L'=L_2-\nabla\Phi\nabla_v$ is essentially m-dissipative on $L^2(\R^d,\mathrm{d}(x,v))$.
\end{proposition}
\begin{proof}
	It holds due to antisymmetry of $v\nabla_x$ that
	\[
	\begin{aligned}
		\|\nabla\Phi\nabla_vf\|_{L^2}^2 &\leq \||\nabla\Phi|\|_\infty^2c_\Sigma^{-1}\left((\nabla_vf,\Sigma\nabla_v f)_{L^2}+ \left(\frac{(v,\Sigma v)}4 f,f\right)_{L^2}-(v\nabla_x f,f)_{L^2} \right)\\
		&\leq \||\nabla\Phi|\|_\infty^2c_\Sigma^{-1} \left( (-L_2f,f)_{L^2} + (d^2C+\frac{dM_\Sigma}2)\|f\|_{L^2}^2\right),
	\end{aligned}
	\]
	analogously to the proof of Proposition~\ref{prop:no-potential}, which again implies
	that the antisymmetric, hence dissipative operator $(\nabla\Phi\nabla_v,D_2)$ is $L_2$-bounded with bound zero.
	This shows the claim.
\end{proof}
Denote by $H_c^{1,\infty}(\R^d)$ the space of functions in $H^{1,\infty}(\R^d)$ with compact support
and set $D'\defeq H_c^{1,\infty}(\R^d)\otimes C_c^\infty(\R^d)$. As $(L',D')$ is dissipative and its closure extends $(L',D_2)$, it is itself essentially m-dissipative. The unitary transformation $T$ from the beginning of this section leaves $D'$ invariant, and it holds that $TL'T^{-1} = L$ on $D'$. This brings us to the first m-dissipativity result for the complete Langevin operator:
\begin{theorem}\label{thm:generator-nice-coeff}
	Let $\Sigma$ satisfy \nref{ass:coeff-growth-one}{($\Sigma$3)} with $\beta\leq -1$ and $\Phi$ be Lipschitz-continuous.
	Then $(L,D)$ with is essentially m-dissipative on $H$.
\end{theorem}
\begin{proof}
	By the previous considerations, $(L,D')$ is essentially m-dissipative on $H$.
	Let $f\in D'$ with $f=g\otimes h$. It holds $g\in H_c^{1,\infty}(\R^d)\subset H^{1,2}(\R^d)$.
	Choose a sequence $(g_n)_{n\in\N}$ with $g_n\in C_c^\infty(\R^d)$, such that $g_n\to g$ in $H^{1,2}(\R^d)$
	as $n\to\infty$. Due to boundedness of $\mathrm{e}^{-\Phi}$ and $v_j\mathrm{e}^{-v^2/2}$ for all $1\leq j\leq d$, it follows immediately that $g_n\otimes h \to f$ and $L(g_n\otimes h)\to Lf$ in $H$ as $n\to\infty$. This extends to arbitrary $f\in D'$ via linear combinations and therefore shows that $C_c^\infty(\R^d)\otimes C_c^\infty(\R^d)$ and hence also $D$, is a core for $(L,D(L))$.
\end{proof}

\subsection{Proof of \texorpdfstring{Theorem~\ref{thm:ess-m-diss}}{Theorem 3.4}}
It is now left to relax the assumptions on $\Sigma$ and $\Phi$ by approximation.
Let the assumptions of Theorem~\ref{thm:ess-m-diss} hold and wlog $\Phi\geq 0$.
For $n\in\N$ we define $\Sigma_n$ via
\[
	\Sigma_n=(a_{ij,n})_{1\leq i,j\leq d}, \quad	a_{ij,n}(v)\defeq a_{ij}\left(\left(\frac{n}{|v|}\wedge 1\right)v\right).
\]
Then each $\Sigma_n$ also satisfies \nref{ass:ellipticity}{($\Sigma$1)}-\nref{ass:coeff-growth-one}{($\Sigma$3)} with $\beta= -1$, since
$\partial_k a_{ij,n}=\partial_k a_{ij}$ on $B_n(0)$ and $|\partial_k a_{ij,n}|\leq \frac{(1+\sqrt{d})nL_{\Sigma,n}}{|v|}$ outside of $\overline{B_n(0)}$, where $L_{\Sigma,n}$ denotes the supremum of $\max_{1\leq k\leq d}|\partial_k a_{ij}|$ on $\overline{B_n(0)}$.
Let further $\eta_m\in C_c^\infty(\R^d)$ for each $m\in\N$ with $\eta=1$ on $B_m(0)$ and set $\Phi_m=\eta_m \Phi$, which is Lipschitz-continuous. Define $H_m$ as $L^2(\R^{2d},\mathrm{e}^{-\frac{v^2}2-\Phi_m(x)}\,\mathrm{d}(x,v))$ and $(L_{n,m},D)$ via
\[
	L_{n,m}f = \sum_{i,j=1}^d a_{ij,n}\partial_{v_j}\partial_{v_i}f + \sum_{i=1}^d \sum_{j=1}^d(\partial_j a_{ij,n}(v)-a_{ij,n}(v)v_j)\partial_{v_i}f + v\cdot\nabla_x f - \nabla\Phi_m\cdot\nabla_v f. 
\]
Then Theorem~\ref{thm:generator-nice-coeff} shows that for each $n,m\in\N$, $(L_{n,m},D)$ is essentially m-dissipative on $H_m$,
and it holds that $L_{n,m}f = Lf$ for all $f\in D$ on $B_m(0)\times B_n(0)$.
Note further that $\|\cdot\|_H\leq \|\cdot\|_{H_m}$.

We need the following estimates:
\begin{lemma}\label{lem:technical-estimate}
	Let $n,m\in\N$ and $\Sigma_n$, $\Phi_m$ as defined above. Then there is a constant $D_1<\infty$
	independent of $n,m$ such that for each $1\leq j\leq d$, the following hold for all $f\in D$:
	\begin{align*}
		\|v_j f\|_{H_m} &\leq D_1 n^{\frac{1+\beta}2} \|(I-L_{n,m})f  \|_{H_m},\\
		\|\partial_{v_j}f\|_{H_m} &\leq D_1 n^{\frac{1+\beta}2}  \|(I-L_{n,m})f  \|_{H_m}.
	\end{align*}
\end{lemma}
\begin{proof}
	Recall the unitary transformations $T_m:L^2(\R^{2d},\mathrm{d}(x,v))\to H_m$ defined by $T_m f=\mathrm{e}^{\frac{v^2}4+\frac{\Phi_m(x)}2} f$,
	as well as the operator $L_{n,m}'=T_m^{-1}L_{n,m}T_m$,
	and let $f\in T_m^{-1}D$. Then
	\[
	\begin{aligned}
		L_{n,m}'f=\sum_{i,j=1}^d a_{ij,n}\partial_{v_j}\partial_{v_i}f&- \frac14(v,\Sigma_n v)f+ \frac12 \operatorname{tr}(\Sigma_n)f+ \sum_{i,j=1}^d \partial_j a_{ij,n}(\frac{v_i}2f+\partial_{v_i}f)\\
		 &-v\nabla_xf+\nabla\Phi_m\nabla_vf.
	\end{aligned}
	\]
	Analogously to the proof of Proposition~\ref{prop:no-potential}
	and due to antisymmetry of $v\nabla_x$ and $\nabla\Phi_m\nabla_v$ on $L^2(\mathrm{d}(x,v))$, it holds that
	\[
	\begin{aligned}
		\| v_j T_mf\|_{H_m}^2 &= \|v_j f\|_{L^2(\mathrm{d}(x,v))}^2 \leq 4 c_{\Sigma}^{-1}\int_{\R^{2d}}\frac14 (v,\Sigma_n v)f^2\,\mathrm{d}(x,v)\\
			&\leq 4c_\Sigma^{-1} \left( (-L_{n,m}'f,f)_{L^2(\mathrm{d}(x,v))} + \int_{\R^{2d}} \frac{f^2}2\left( \operatorname{tr}(\Sigma_n)+\sum_{i,j=1}^d\partial_ja_{ij,n}v_i\right)\mathrm{d}(x,v) \right).
	\end{aligned}
	\]
	Since $|\operatorname{tr}(\Sigma_n)|\leq |\operatorname{tr}(\Sigma)|\leq d\cdot M_\Sigma$ and
	\[
		|\partial_j a_{ij,n}(v)v_i| \leq |\partial_j a_{ij}(v)|\cdot |v_i| \leq \max\{ B_\Sigma, M\cdot n^{\beta+1}\}
		\quad\text{ for all }v\in B_n(0),
	\]
	as well as
	\[
		|\partial_j a_{ij,n}(v)v_i| \leq (1+\sqrt{d})n\frac{|v_i|}{|v|} \max_{1\leq k\leq d}\sup_{y\in B_n(0)}|\partial_k a_{ij}(y)| \leq 2\sqrt{d}Mn^{\beta+1}\quad\text{ for all } v\notin B_n(0),	
	\]
	and wlog $B_\Sigma\leq M\cdot n^{\beta+1}$, it follows that
	\[
		\| v_j T_mf\|_{H_m}^2 \leq 4c_\Sigma^{-1} (-L_{n,m}'f,f)_{L^2(\mathrm{d}(x,v))} + 2c_\Sigma^{-1}(dM_\Sigma + 2d^{5/2}Mn^{\beta + 1})\|f\|_{L^2(\mathrm{d}(x,v))}^2.
	\]
	Further, it clearly holds that
	\begin{align*}
		(-L_{n,m}'f,f)_{L^2(\mathrm{d}(x,v))} &\leq \frac14 \left(\|L_{n,m}'f \|_{L^2(\mathrm{d}(x,v))}+\|f\|_{L^2(\mathrm{d}(x,v))}\right)^2 \quad\text{ and }\\
		\|f\|_{L^2(\mathrm{d}(x,v))}^2&\leq \left(\|L_{n,m}'f\|_{L^2(\mathrm{d}(x,v))}+\|f\|_{L^2(\mathrm{d}(x,v))}\right)^2.
	\end{align*}
	Dissipativity of $(L_{n,m}',T_m^{-1}D)$ on $L^2(\mathrm{d}(x,v))$ implies
	\[
		\|L_{n,m}'f\|_{L^2(\mathrm{d}(x,v))}+\|f\|_{L^2(\mathrm{d}(x,v))}\leq \|(I-L_{n,m}')f\|_{L^2(\mathrm{d}(x,v))}+2\|(I-L_{n,m}')f\|_{L^2(\mathrm{d}(x,v))}.
	\]
	Overall, we get
	\[
	\begin{aligned}
		\| v_j T_mf\|_{H_m}^2 &\leq 2c_\Sigma^{-1}(1+2(dM_\Sigma + 2d^{5/2}Mn^{\beta + 1})) \|(I-L_{n,m}')f\|_{L^2(\mathrm{d}(x,v))}^2\\
		 &\leq 18c_\Sigma^{-1} d^3n^{\beta + 1}\max\{ M_\Sigma, M\}  \|(I-L_{n,m}')f\|_{L^2(\mathrm{d}(x,v))}^2.
	\end{aligned}
	\]
	Since
	\[
		\|(I-L_{n,m}')f\|_{L^2(\mathrm{d}(x,v))}^2 = \|T_m^{-1}(I-L_{n,m})T_mf\|_{L^2(\mathrm{d}(x,v))}^2 = \|(I-L_{n,m})T_mf\|_{H_m}^2,
	\]
	this proves the first statement with $D_1= \sqrt{18c_\Sigma^{-1}d^3\max\{ M_\Sigma,M \}}$.
	
	For the second part, note that $\partial_{v_j}T_mf= T_m\partial_{v_j}f+\frac{v_j}2T_mf$ and that
	\[
	\begin{aligned}
		\| T_m\partial_{v_j}f\|_{H_m}^2 &= (\partial_{v_j}f,\partial_{v_j}f)_{L^2(\mathrm{d}(x,v))}^2
		\leq c_\Sigma^{-1}\int_{\R^{2d}} (\nabla_v f,\Sigma_n\nabla_v f)_\mathrm{euc}\,\mathrm{d}(x,v)\\
		&\leq c_\Sigma^{-1} \left( (-L_{n,m}'f,f)_{L^2} + \int_{\R^{2d}} \frac12 \operatorname{tr}(\Sigma_n)f^2+\sum_{i,j=1}^d\partial_ja_{ij,n}\frac{v_i}2f^2\,\mathrm{d}(x,v) \right).
	\end{aligned}
	\]
	Repeating all calculations of the first part yields
	\[
		\|\partial_{v_j}T_mf\|_{H_m}\leq \left(\frac{D_1}2+\frac{D_1}2\right)n^{1+\beta}\|(I-L_{n,m})T_mf\|_{H_m}.
	\]
\end{proof}

Fix some pure tensor $g\in C_c^\infty(\R^d)\otimes C_c^\infty(\R^d)$. We prove that for every $\eps>0$, we can find some $f\in D$ such that $\|(I-L)f-g\|_H<\eps$. This then extends to arbitrary $g\in C_c^\infty(\R^d)\otimes C_c^\infty(\R^d)$ via linear combinations and therefore implies essential m-dissipativity of $(L,D)$ on $H$, since $C_c^\infty(\R^d)\otimes C_c^\infty(\R^d)$ is dense in $H$. If $\beta\leq -1$, then the proof is easier and follows analogously to the proof of of \cite[Theorem 6.3.1]{Conrad2011}. Therefore we will assume $\beta>-1$. Recall that in this case, we have $|\nabla\Phi(x)|\leq N(1+|x|^\gamma)$ for all $x\in\R^d$, where $\gamma<\frac{2}{1+\beta}$, see the assumptions of Theorem~\ref{thm:ess-m-diss}.

Denote the support of $g$ by $K_x\times K_v$, where $K_x$ and $K_v$ are compact sets in $\R^d$.
By a standard construction, for each $\delta_x,\delta_v >0$,
there are smooth cutoff functions $0\leq \phi_{\delta_x},\psi_{\delta_v}\leq 1\in C_c^\infty(\R^d)$
with $\supp(\phi_{\delta_x})\subset B_{\delta_x}(K_x)$,
$\supp(\psi_{\delta_v})\subset B_{\delta_v}(K_v)$,
$\phi_{\delta_x}=1$ on $K_x$, $\psi_{\delta_v}=1$ on $K_v$.
Moreover, there are constants $C_\phi, C_\psi$ independent of $\delta_x$ and $\delta_v$ such that
\[
	\|\partial^s\phi_{\delta_x}\|_\infty \leq C_\phi \delta_x^{-|s|}
	\quad\text{ and }\quad
	\|\partial^s\psi_{\delta_v}\|_\infty \leq C_\psi \delta_v^{-|s|}
\]
for all multi-indices $s\in\N^d$. Fix $\alpha$ such that $\frac{1+\beta}2 < \alpha < \frac1\gamma$.
For any $\delta>0$, we set $\delta_x\defeq \delta^\alpha$ and $\delta_v\defeq\delta$, and then define
$\chi_\delta(x,v)\defeq \phi_{\delta_x}(x)\psi_{\delta_v}(v)=\phi_{\delta^\alpha}(x)\psi_{\delta}(v)$.

For $f\in D$, $\delta>0$, consider $f_\delta\defeq \chi_\delta f$, which is an element of $D$, as $\chi_\delta\in D$.
Without loss of generality, we consider $\delta$ and hence $\delta^\alpha$
sufficiently large such that $\supp(\phi_{\delta^\alpha})\subset B_{2\delta^\alpha}(0)$,
$\supp(\psi_{\delta})\subset B_{2\delta}(0)$ and that there are $n,m\in \N$ that satisfy
\begin{equation}\label{eq:container_balls}
\supp(\phi_{\delta^\alpha})\times \supp(\psi_{\delta}) \subset B_m(0)\times B_n(0)\subset B_{2\delta^\alpha}(0)\times B_{2\delta}(0). 
\end{equation}
The following then holds:
\begin{lemma}\label{lem:even-more-technical}
	Let $g\in C_c^\infty(\R^d)\otimes C_c^\infty(\R^d)$ and $\phi,\psi$ as above.
	Then there is a constant $D_2<\infty$ and a function $\rho:\R\to\R$ satisfying $\rho(s)\to 0$ as $s\to\infty$,
	such that for any $\delta$, $n$ and $m$ satisfying \eqref{eq:container_balls},
	\[
		\|(I-L)f_\delta-g\|_H \leq \|(I-L_{n,m})f -g\|_{H_m} + D_2\cdot\rho(\delta) \|(I-L_{n,m})f\|_{H_m}
	\]
	holds for all $f\in D$.
\end{lemma}
\begin{proof}
By the product rule,
\[
\begin{aligned}
	\|(I-L)f_\delta-g\|_H &\leq \|\chi_\delta((I-L)f -g) \|_H + \sum_{i,j=1}^d \| a_{ij}\phi_{\delta^\alpha}(x)\partial_j\partial_i\psi_\delta(v) f \|_H\\
		&+2\sum_{i,j=1}^d \|a_{ij}\phi_{\delta^\alpha}(x)\partial_i\psi_\delta(v)\partial_{v_j}f\|_H
		 + \sum_{i,j=1}^d \|\partial_j a_{ij}\phi_{\delta^\alpha}(x)\partial_i\psi_\delta(v)f\|_H \\
		&+\sum_{i,j=1}^d \| a_{ij}v_j\phi_{\delta^\alpha}(x)\partial_i\psi_\delta(v)f\|_H
		 + \sum_{i=1}^d \| v_i\partial_i\phi_{\delta^\alpha}(x)\psi_\delta(v)f\|_H \\
		&+\sum_{i=1}^d \| \partial_i\Phi \phi_{\delta^\alpha}(x)\partial_i\psi_\delta(v)f\|_H.
\end{aligned}
\]
Due to the choice of $n$ and $m$, every $\|\cdot \|_H$ on the right hand side can be replaced with $\|\cdot \|_{H_m}$,
$a_{ij}$ by $a_{ij,n}$, and $\Phi$ by $\Phi_m$, hence $L$ by $L_{n,m}$.

We now give estimates for each summand of the right hand side, in their order of appearance:
\begin{enumerate}[(1)]
	\item $\|\chi_\delta((I-L)f -g) \|_H \leq \|(I-L_{n,m})f -g\|_{H_m}$,
	\item $\|a_{ij}\phi_{\delta^\alpha}(x)\partial_j\partial_i\psi_\delta(v) f \|_H \leq M_\Sigma C_\psi \delta^{-2}\|f\|_{H_m}$,
	\item $\|a_{ij}\phi_{\delta^\alpha}(x)\partial_i\psi_\delta(v)\partial_{v_j}f\|_H \leq  M_\Sigma C_\psi \delta^{-1}\|\partial_{v_j}f\|_{H_m}$,
	\item $\|\partial_j a_{ij}\phi_{\delta^\alpha}(x)\partial_i\psi_\delta(v)f\|_H \leq \max\{B_\Sigma,M\cdot (2\delta)^{\beta\vee 0}\} C_\psi \delta^{-1}\|f\|_{H_m}$,
	\item $\|a_{ij}v_j\phi_{\delta^\alpha}(x)\partial_i\psi_\delta(v)f\|_H \leq M_\Sigma C_\psi \delta^{-1}\|v_jf\|_{H_m}$,
	\item $\|v_i\partial_i\phi_{\delta^\alpha}(x)\psi_\delta(v)f\|_H \leq C_\phi\delta^{-\alpha}\|v_if\|_{H_m}$,
	\item $\|\partial_i\Phi \phi_{\delta^\alpha}(x)\partial_i\psi_\delta(v)f\|_H \leq N(1+(2\delta^\alpha)^\gamma) C_\psi\delta^{-1} \|f\|_{H_m}$,
\end{enumerate}
where the last inequality is due to $|\partial_i\Phi(x)|\leq N(1+|x|^\gamma)$ for all $x\in\R^d$ and the support of the cutoff as in \eqref{eq:container_balls}.
Application of Lemma~\ref{lem:technical-estimate} shows the existence of $D_2$ independent of $n,m$, such that
\[
	\|(I-L)f_\delta-g\|_H \leq \|(I-L_{n,m})f -g\|_{H_m} + D_2\cdot\rho(\delta) \|(I-L_{n,m})f\|_{H_m}
\]
where
\[
	\rho(\delta)\defeq \delta^{-2}
	+ 2^{\frac{1+\beta}2}\delta^{\frac{1+\beta}2-1}
	+ 2^{\beta\vee 0}\delta^{(\beta\vee 0)-1}
	+ 2^{\frac{1+\beta}2}\delta^{\frac{1+\beta}2-\alpha}
	+ \delta^{-1}
	+ 2^\gamma\delta^{\alpha\gamma-1}.
\]
Clearly $\rho(\delta)\to 0$ as $\delta\to\infty$ due to $\beta<1$ and the definition of $\alpha$.
\end{proof}

Now finally we show that for each $\eps>0$, we can find some $f_\delta\in D$ such that
\[\|(I-L)f_\delta-g\|_H <\eps.\]
Choose $\delta>0$ large enough such that $\rho(\delta) < \frac{\eps}{4D_2\|g\|_H}$ (where $\rho$ ans $D_2$ are provided by Lemma~\ref{lem:even-more-technical}) and that there exist $n,m$ satisfying \eqref{eq:container_balls}.

Then choose $f\in D$ via Theorem~\ref{thm:generator-nice-coeff} such that $\|(I-L_{n,m})f-g\|_{H_m}<\min\{ \frac\eps 2,\|g\|_H \}$
and define $f_\delta$ as before.
Note that due to the choice of the cutoffs, it holds $\|g\|_H=\|g\|_{H_m}$, therefore
\[
	\|(I-L)f_\delta-g\|_H < \frac\eps 2 + \frac\eps{4\|g\|_{H_m}}(\|(I-L_{n,m})f-g\|_{H_m}+\|g\|_{H_m}) < \eps.
\]
As mentioned earlier, this shows essential m-dissipativity of the operator $(L,D)$ on $H$ and therefore concludes the proof of Theorem~\ref{thm:ess-m-diss}.

\section{Applications}

\subsection{The associated Cauchy problem}\label{subsec:cauchy}
We consider the abstract Cauchy problem associated with the operator $L$. Given the initial condition $u_0\in H$,
$u:[0,\infty)\to H$ should satisfy
\begin{equation}\label{eq:cauchy-kol}
	\partial_t u(t) = \left(\tr\left(\Sigma H_v \right)+ b\cdot\nabla_v +v\cdot\nabla_x -\nabla\Phi\cdot\nabla_v\right)u(t) \quad\text{ and }\quad u(0)=u_0.
\end{equation}
If we set $u(t)\defeq T_tu_0$, where $(T_t)_{t\geq 0}$ is the semigroup on $H$ generated by the closure $(L,D(L))$ of $(L,D)$,
then the map $t\mapsto u(t)$ is continuous in $H$. For all $t\geq 0$, it holds that $\int_0^t u(s)\,\mathrm{d}s\in D(L)$ with
$L\int_0^t u(s)\,\mathrm{d}s=T_tu_0-u_0 = u(t)-u_0$, hence $u$ is the unique mild solution to the abstract Cauchy problem.

If $u_0\in D(L)$, then $u(t)\in D(L)$ for all $t\geq 0$, and $\partial_t u(t) = LT_tu_0=Lu(t)$, so $u$ is even a classical solution to the abstract Cauchy problem associated to $L$. In particular, this holds for all $u_0 \in C_c^2(\R^{d\times d})$, since $L$ is dissipative there and it extends $D$, which implies $C_c^2(\R^{d\times d})\subset D(L)$.

In this context, Theorem~\ref{thm:main-result} shows exponential convergence of the unique solution $u(t)$ to a constant as $t\to\infty$. More precisely, for each $\theta_1 > 1$ we can calculate $\theta_2\in(0,\infty)$ depending on the choice of $\Sigma$ and $\Phi$ such that for all $t\geq 0$,
\[
	\left\| u(t)-\int_E u_0\,\mathrm{d}\mu \right\|_H \leq \theta_1\mathrm{e}^{-\theta_2 t} \left\|u_0-\int_E u_0\,\mathrm{d}\mu \right\|_H.
\]

\subsection{Connection to Langevin dynamics with multiplicative noise}\label{subsec:stochastics}

So far, our considerations have been purely analytical, giving results about the core property of $D$ for $L$
and rate of convergence for the generated semigroup $(T_t)_{t\geq 0}$ in $H$. However, this approach is still quite natural in the context of the Langevin SDE \eqref{eq:sde}, as the semigroup has a meaningful stochastic representation. The connection is achieved via the powerful theory of generalized Dirichlet forms as developed by Stannat in \cite{GDF}, which gives the following:

Assume the context of Theorem~\ref{thm:ess-m-diss}. There exists a Hunt process
\[
	\mathbf{M}=\left(\Omega,\mathcal{F},(\mathcal{F}_t)_{t\geq 0},(X_t,V_t),(P_{(x,v)})_{(x,v)\in\R^d\times\R^d}\right)
\]
with state space $E=\R^d\times\R^d$, infinite lifetime and continuous sample paths ($P_{(x,v)}$-a.s. for all $(x,v)\in E$),
which is properly associated in the resolvent sense with $(T_t)_{t\geq 0}$. In particular (see \cite[Lemma 2.2.8]{Conrad2011}), this means that for each bounded measurable $f$ which is also square-integrable with respect to the invariant measure $\mu$ and all $t>0$, $T_tf$ is a $\mu$-version of $p_tf$, where $(p_t)_{t\geq 0}$ is the transition semigroup of $\mathbf{M}$ with
\[
	p_tf: \R^d\times\R^d\to \R,\qquad (x,v)\mapsto \mathbb{E}_{(x,v)}\left[f(X_t,V_t) \right].
\]
This representation can be further extended to all $f\in H$, see for example \cite[Exercise IV.2.9]{MaRockner}.
Moreover, if $\mu$-versions of $\Sigma$ and $\Phi$ are fixed, then $P_{(x,v)}$ solves the martingale problem for $L$ on $C_c^2(E)$ for $L$-quasi all $(x,v)\in E$, i.e. for each $f\in C_c^2(E)$, the stochastic process $(M_t^{[f]})_{t\geq0}$ defined by
\[
	M_t^{[f]}\defeq f(X_t,V_t)-f(X_0,V_0)-\int_0^t Lf(X_s,V_s)\,\mathrm{d}s,
\]
is a martingale with respect to $P_{(x,v)}$. If $h\in L^2(\mu)$ is a probability density with respect to $\mu$, then the law $P_h\defeq \int_E P_{(x,v)}h(x,v)\,\mathrm{d}\mu$ solves the martingale problem for $(L,D(L))$, without the need to fix specific versions of $\Sigma$ and $\Phi$. In particular, this holds for $h=1$. As in \cite[Lemma 2.1.8]{Conrad2011}, for $f\in D(L)$ with $f^2\in D(L)$ and $Lf\in L^4(\mu)$, a martingale is also defined via
\[
	N_t^{[f]}\defeq (M_t^{[f]})^2 - \int_0^t L(f^2)(X_s,V_s)-(2fLf)(X_s,V_s)\,\mathrm{d}s,\qquad t\geq 0,
\]
which may serve as a way to verify that $\mathbf{M}$ is already a weak solution of \eqref{eq:sde},
as it allows a representation of the quadratic variation process. Indeed, if we set $f_n^i(x,v)\defeq \varphi_n(x_i)x_i$
for a suitable sequence $(\varphi_n)_{n\in\N}$ of cutoff functions as in Definition~\ref{def:cutoff}, evaluation of $N_t^{[f_n^i]}$ shows that the quadratic variation $[M^{[f_n^i]}]_t$ of $M_t^{[f_n^i]}$ is constantly zero, which implies the same for $M_t^{[f_n^i]}$. Hence, by introducing appropriate stopping times, it follows that $X_t^i-X_0^i= \int_0^t V_s^i\,\mathrm{d}s$, so the first line of the SDE \eqref{eq:sde} is satisfied.

In an analogous procedure, using $g_n^i(x,v)\defeq \varphi_n(v_i)v_i$, we can see that the quadratic covariation $[V^i,V^j]_t$ is given by $2\int_0^t a_{ij}(V_s)\,\mathrm{d}s$. Since $\Sigma$ is strictly elliptic, the diffusion matrix $\sigma$ is invertible and by Lévy's characterization, the process $B_t\defeq \int_0^t \frac1{\sqrt{2}}\sigma^{-1}(V_s)\,\mathrm{d}M_s$ is a standard $d$-dimensional Brownian motion, where $M_t\defeq (M_t^{[v_1]},\dots,(M_t^{[v_d]})$, which is a local martingale. Moreover, it holds that
\[
	\mathrm{d}V_t= \mathrm{d}M_t+  b(V_t)-\nabla\Phi(X_t)\,\mathrm{d}t
		= \sqrt{2}\sigma(V_t)\mathrm{d}B_t + b(V_t)-\nabla\Phi(X_t)\,\mathrm{d}t,
\]
so $(X_t,V_t)$ is a weak solution to the SDE \eqref{eq:sde} with initial distribution $h\mu$ under $P_h$.

Finally, in this context, the statement on hypocoercivity (Theorem~\ref{thm:main-result}) shows that for every $\theta_1>1$, there is an explicitly computable $\theta_2\in (0,\infty)$ depending on the choice of $\Sigma$ and $\Phi$, such that the transition semigroup $(p_t)_{t\geq 0}$ satisfies
\begin{equation}\label{eq:prob-conv-speed}
	\|p_tg-\int_E g\,\mathrm{d}\mu\|_{L^2(\mu)} \leq \theta_1\mathrm{e}^{-\theta_2 t}\|g-\int_E g\,\mathrm{d}\mu  \|_{L^2(\mu)}
\end{equation}
for all $g\in L^2(\mu)$ and $t\geq 0$. In particular, this implies that the probability law $P_\mu$ on the space of continuous paths on $E$ with initial distribution (and invariant measure)  $\mu$
has the strong mixing property, i.e. for any Borel sets $A_1, A_2$ on the path space,
it holds that
\[
	P_\mu(\varphi_tA_1\cap A_2) \to P_\mu(A_1) P_\mu(A_2)\quad\text{ as }t\to\infty,
\]
where $\varphi_tA_1\defeq \{(Z_s)_{s\geq 0}\in C([0,\infty),E)\mid (Z_{s+t})_{s\geq 0}\in A_1  \}$.
This follows from \eqref{eq:prob-conv-speed} and associatedness of the semigroups to the probability law $P_\mu$, see for example \cite[Remark 2.1.13]{Conrad2011}.

\subsection{Corresponding Fokker-Planck equation}\label{subsec:fokker-planck}
In this part we give a reformulation of the convergence rate result detailed in Section~\ref{subsec:cauchy}
for readers which are more familiar with the classical Fokker-Planck formulation for probability densities.
In the current literature, Fokker-Planck equations are more often expressed as equations on measures, rather than functions. For example, in the non-degenerate case, exponential convergence in total variation to a stationary solution is studied in \cite{BRS}, which includes further references to related works. Our goal here however is simply to make the convergence result immediately applicable to less specialized readers in the form of the estimate \eqref{eq:fokker-planck-conv} for solutions to the Cauchy problem associated with the operator defined in \eqref{eq:fokker-planck-op}, hence we stick to the expression via probability densities.

Given a Kolmogorov backwards equation of the form $-\partial_t u(x,t)= L^\mathrm{K}u(x,t)$,
the corresponding Fokker-Planck equation is given by $\partial_t f(x,t)=L^\mathrm{FP}f(x,t)$, where $L^\mathrm{FP}=(L^\mathrm{K})'$
is the adjoint operator of $L^\mathrm{K}$ in $L^2(\R^d,\mathrm{d}x)$, restricted to smooth functions. In our setting, $L^\mathrm{K}=L$ produces via integration by parts for $f\in D$:
\begin{equation}\label{eq:fokker-planck-op}
	L^\mathrm{FP}f = \sum_{i,j=1}^d \partial_{v_i}(a_{ij} \partial_{v_j}f + v_j a_{ij} f) -v\cdot\nabla_xf+\nabla\Phi\nabla_vf.
\end{equation}
Consider the Fokker-Planck Hilbert space $\widetilde{H}\defeq L^2(E,\widetilde{\mu})$, where
\[
	\widetilde{\mu}\defeq  (2\pi)^{-\frac{d}2} \mathrm{e}^{\Phi(x)+\frac{v^2}2}\,\mathrm{d}x\otimes\mathrm{d}v.
\]
Then a unitary Hilbert space transformation between $H$ and $\widetilde{H}$ is given by
\[
	T:H\to \widetilde{H},\quad  Tg=\rho g \quad\text{ with }\quad \rho(x,v)\defeq \mathrm{e}^{-\Phi(x)-\frac{v^2}2}.
\]
Let $(T_t)_{t \geq 0}$ be the semigroup on $H$ generated by $(L,D(L))$ and denote by $(T_t^*)_{t\geq 0}$ and $L^*$ the adjoint semigroup on $H$ and its generator, respectively. It is evident that for $f\in D$, $L^*$ is given as $L^*f = (S+A)f$, where $S$ and $A$ refer to the symmetric and antisymmetric components of $L$ respectively, as defined in Definition~\ref{def:operators}. As mentioned in \ref{rem:magic}, we achieve the exact same results for the equation corresponding to $L^*$ as for the one corresponding to $L$, which we considered in Section 3. In particular, $(L^*,D)$ is essentially m-dissipative and its closure $(L^*,D(L^*))$ generates $(T_t^*)_{t\geq 0}$, which converges exponentially to equilibrium with the same rate as $(T_t)_{t\geq 0}$.

Let $\widetilde{T}_t g \defeq T(T_t^*)T^{-1}g$ for $t\geq 0$, $g\in\widetilde{H}$. Then
$(\widetilde{T}_t)_{t\geq 0}$ is a strongly continuous contraction semigroup on $\widetilde{H}$ with the generator $(TL^*T^{-1},T(D(L^*)))$. It is easy to see that $L^\mathrm{FP}=TL^*T^{-1}$, so for each initial condition $u_0\in\widetilde{H}$, $u(t)\defeq\widetilde{T}_t u_0$ is a mild solution to the Fokker-Planck Cauchy problem. Note that for $\Phi\in C^\infty(\R^d)$, the transformation $T$ leaves $D$ invariant, which implies $D\subset T(D(L^*))$ and essential m-dissipativity of $(L^{\mathrm{FP}},D)$ on $\widetilde{H}$.

If $u_0\in T(D(L^*))$, then
\[
	\partial_t \widetilde{T}_t u_0 = T(L^*T_t^*)T^{-1}u_0,
\]
and therefore
\[
\begin{aligned}
	\int_E \partial_t u(t) f\,\mathrm{d}(x,v)
		&= \int_E L^*T_t^*T^{-1}u_0 f \,\mathrm{d}\mu
		= \int_E T_t^*T^{-1}u_0 Lf\,\mathrm{d}\mu \\
		&= \int_E TT_t^*T^{-1}u_0 Lf\,\mathrm{d}(x,v)
		= \int_E L^\mathrm{FP}u(t) f\,\mathrm{d}(x,v),
\end{aligned}
\] 
so $u(t)$ is also a classical solution.
Due to the invariance of $\mu$ for $L$, a stationary solution is given by $\rho$ and by Theorem~\ref{thm:main-result}, for every $\theta_1> 1$ and the appropriate $\theta_2$ it holds that
\begin{equation}\label{eq:fokker-planck-conv}
\begin{aligned}
	\left\| u(t)-\rho(u_0,\rho)_{\widetilde{H}} \right\|_{\widetilde{H}}
	&= \left\| T_t^*T^{-1}u_0 - (T^{-1}u_0,1)_H  \right\|_H \\
	&\leq \theta_1\mathrm{e}^{-\theta_2 t} \left\| T^{-1}u_0 - (T^{-1}u_0,1)_H  \right\|_H \\
	&= \theta_1\mathrm{e}^{-\theta_2 t} \left\| u_0-\rho(u_0,\rho)_{\widetilde{H}} \right\|_{\widetilde{H}}.
\end{aligned}
\end{equation}

This shows exponential convergence to a stationary state for solutions to the Fokker-Planck equation.

\bibliographystyle{spphys}
\bibliography{sources}

\end{document}